\documentclass[a4paper]{amsart}

\usepackage{color}
\usepackage{amsxtra}
\usepackage{mathtools}
\usepackage{enumerate}
\usepackage{nicefrac}
\mathtoolsset{showonlyrefs} 
\usepackage{ mathrsfs }
\usepackage{tikz}
\usepackage{stmaryrd}
\usetikzlibrary{shadings}
\usepackage{color}
\usepackage{hyperref}
\hypersetup{
  colorlinks   = true, 
  urlcolor     = blue, 
  linkcolor    = blue, 
  citecolor   = red 
}

\newtheorem{theorem}{Theorem}[section]
\newtheorem{lemma}{Lemma}[section]

\newtheorem{definition}{Definition}[section]
\newtheorem{corollary}{Corollary}

\newcommand{\lpo}{ L\!\prescript{p}{}(\Omega)}
\newcommand{\lpr}{ L\!\prescript{p}{}(\mathbb{R}^N)}
\newcommand{\lir}{ L\!\prescript{\infty}{}(\mathbb{R}^N)}
\newcommand{\lps}{L\!\prescript{p_s^*}{}(\mathbb{R}^N)}
\newcommand{\lpsn}{L\!\prescript{\nicefrac{N}{sp}}{}(\mathbb{R}^N)}

\newcommand{\wspo}{W\!\prescript{s,p}{}(\Omega)}
\newcommand{\twspo}{\widetilde{W}\!\prescript{s,p}{}(\Omega)}
\newcommand{\wspr}{W\!\prescript{s,p}{}(\mathbb{R}^N)}
\newcommand{\wsprc}{{\mathcal{W}}\!\prescript{s,p}{}(\mathbb{R}^N)}
\newcommand{\cir}{C\!\prescript{\infty}{}(\mathbb{R}^N)}
\newcommand{\cio}{C\!\prescript{\infty}{}(\Omega)}
\newcommand{\ccir}{C_0\!\!\!\prescript{\infty}{}(\mathbb{R}^N)}
\newcommand{\cchi}{\protect\raisebox{2pt}{$\chi$}}

\title[Spectrum of the fractional $p-$Laplacian]{
Spectrum of the fractional $p-$Laplacian in $\mathbb{R}^N$
and decay estimate for positive solutions of a Schr\"odinger
equation
}

\author[L. M. Del Pezzo]{Leandro M. Del Pezzo}
	\address{Leandro M. Del Pezzo \hfill\break\indent
		CONICET -- UTDT \hfill\break\indent
		Departamento de Matem\'aticas y 
		Estad\'istica
		\hfill\break\indent Av. Figueroa Alcorta 7350 (C1428BCW)
		\hfill\break\indent Buenos Aires, ARGENTINA. }
	\email{ldelpezzo@utdt.edu}
	\urladdr{http://cms.dm.uba.ar/Members/ldpezzo/}

\author[A. Quaas]{Alexander Quaas}
	\address{A. Quaas
        \hfill\break\indent Departamento de Matem\'atica, 
        \hfill\break\indent Universidad T\'ecnica Federico Santa Mar\'ia 
        \hfill\break\indent Casilla V-110, Avda. 
        \hfill\break\indent Espa\~na, 1680 -- Valpara\'iso, CHILE.}
    \email{alexander.quaas@usm.cl}

\begin{document}

\begin{abstract}
    In this paper, we prove the existence of unbounded sequence of
    eigenvalues for the fractional $p-$Laplacian with weight in
    $\mathbb{R}^N.$ We also show a nonexistence result when the weight 
    has positive integral. 
    
    In addition, we show some qualitative properties of the
    first eigenfunction including a sharp decay estimate. Finally, 
    we extend the decay result to the positive solutions of a
    Schr\"odinger type equation.

\end{abstract}
\maketitle

\section{Introduction and main results}

	In this paper, we study the following eigenvalue problem  
	 \begin{equation}\label{eq:ep}
	            (-\Delta_p)^s u(x)=\lambda g(x)|u(x)|^{p-2}u(x)
	                    \quad \text{ in }\mathbb{R}^N,
    \end{equation}
	where $0<s<1,$  $p>1,$ and $g$  is a weight function satisfying 
	some conditions to be specified later.
	Here  $(-\Delta_p)^s$ denotes the fractional $p-$Laplacian 	
	operator, that is
	\[
		(-\Delta_p)^s u(x)\coloneqq2\mbox{ P.V.}\int_{\mathbb{R}^N}
		\frac{|u(x)-u(y)|^{p-2}(u(x)-u(y))}{|x-y|^{N+sp}}\, dy,
	\]
	where P.V. is a commonly used abbreviation for 
	“in the principal value sense”.
	
	\medskip

	Before we describe our principal results, we will give some motivations. 
	Nonlocal equation of $p$-Laplace type where introduce in 
	\cite{MR2471902,BCF,Caffarelli2012,IN}.
	Fractional Sobolev spaces semi-norms 
	(see for an introduction to the topic in \cite{DNPV} and below)  
	are  natural in the weak form and\ or functional associated 
	with the operator  $(-\Delta_p)^s$, 
	therefore eigenvalues can be studied in bounded domains using variational 
	methods  see \cite{LL,MR3556755,FP,MR3552458,MR3411543}.
	
	\medskip

	All spectrum in bounded domain for the fractional Laplacian ($p=2$) 
	is studied in \cite{MR3002745}, see also \cite{MR3089742}. 
	The variational unbounded sequence of eigenvalues of the fractional 
	$p-$Laplacian is studied in \cite{MR3411543}. 
	
	\medskip
	
	In unbounded domain, in particular $\mathbb{R}^N$, some weight function 
	with some condition needs to be introduce, moreover non-existence may also 
	appear. Spectrum in $\mathbb{R}^N$ for local problems with weights 
	are studied in \cite{MR1007489,MR1098396,MR1364493,MR1336957,MR1356326}. 
	As far as we know, there isn't an extension of these type of results for 
	the nonlocal setting even when $p=2$. Therefore, one of the main purpose 
	of this work is study the spectrum in $\mathbb{R}^N$ in a nonlocal setting.
	
	Finally, let observe that the eigenvalues are a starting point in study some 
	type bifurcation results in  $\mathbb{R}^N$, 
	see for example \cite{MR1390979} and \cite{MR1412438} in the local case. 
	Other type of bifurcation results in $\mathbb{R}^N$ for the fractional 
	Laplacian can be found in \cite{MR3635980}. 
	While in \cite{MR3556755}, the authors show a bifurcation results 
	in bounded domain for the fractional $p-$Laplacian.

	\bigskip
 
	Now we will describe our results. As in the local case, we will split 
	the discussion in to two cases $sp<N$ and  $sp\geq N,$ 
	where different approaches are needed. But first we need to introduce 
	the theoretical framework for them. 
	
	The fractional Sobolev spaces $\wspr$
	is defined to be the set of functions $u\in\lpr$ such that
		\[
			\llbracket u\rrbracket_{s,p}^p\coloneqq
			\int_{\mathbb{R}^N}\int_{\mathbb{R}^N}
			\dfrac{|u(x)-u(y)|^p}{|x-y|^{N+sp}}\, dxdy<\infty.
		\]
	While, the closure of $\ccir$ with respect to the
		norm $\llbracket \cdot\rrbracket_{s,p}$ is denoted by $\wsprc.$
		For more details about the spaces, see Section \ref{NyP}.

    \begin{definition}\label{def:eigen} 
        Let $g\in\lir$ be such that $g\not\equiv0.$ 
        \begin{enumerate}[(i)]
	        \item If $sp<N$ and $g\in\lpsn$  we say that 
	            a pair $(u,\lambda)\in\wsprc\times\mathbb{R}$ 
	            is a weak solution of \eqref{eq:ep}
                if
                \begin{equation}\label{eq:ws}
	                \begin{aligned}
                        \int_{\mathbb{R}^N}\int_{\mathbb{R}^N}
                         &
                         \dfrac{|u(x)-u(y)|^{p-2}(u(x)-u(y))
                         (\varphi(x)-\varphi(y))}
                         {|x-y|^{N+ps}}\,dxdy\\
                        &\hspace{4cm}=
                        \lambda \int_{\mathbb{R}^N}g(x)|u(x)|^{p-2}u(x)\varphi(x) dx,
                     \end{aligned}
                \end{equation}              
                for all $\varphi\in\ccir.$ 
            \item If $sp\ge N,$  we say that 
	            a pair $(u,\lambda)\in\wspr\times\mathbb{R}$ is a weak solution of
	            \eqref{eq:ep}
                if $\eqref{eq:ws}$ holds for all $\varphi\in\ccir.$ 
            \item In both cases, a pair $(u,\lambda)$ 
                is called an eigenpair, in which case, 
                $\lambda$ is called an eigenvalue and  $u$ a
                corresponding eigenfunction.             
        \end{enumerate}        
    \end{definition}

   	Lastly,
    $g_+(x)\coloneqq\max\{g(x),0)\}$ and $g_{-}(x)\coloneqq\max\{-g(x),0)\}.$

	\medskip

	Our first aim is extended the results for the local case given by 
	\cite{MR1356326, MR1364493} to the nonlocal case.
	Let start with the case $sp<N$.    
    \begin{theorem}\label{thm:existence1} 
    	Assume that $sp<N$ and $g\in \lir\cap\lpsn.$ 
        \begin{enumerate}[(i)]
	        \item If $g_+\not\equiv0$ then there exists a sequence of eigenpairs 
	        $\{(u_n,\lambda_n(g))\}_{n\in\mathbb{N}}$ such that 
	        \[
	            \int_{\mathbb{R}^n} g(x)|u_n(x)|^p\,dx=1\quad\forall 
	            n\in\mathbb{N},
	        \]
	        and
	        \[
	            0<\lambda_1(g)<\lambda_2(g)\le\cdots\le\lambda_n(g)\to\infty.
	        \]
	        Moreover
	        \[
	            \lambda_1(g)=\min\left\{\llbracket u\rrbracket_{s,p}^p\colon 
	            u\in \wsprc \text{ and }\int_{\mathbb{R}^n} g(x)|u(x)|^p
	            \,dx=1\right\},
	        \]
	        is a simple eigenvalue with constant sign eigenfunction.
	        
	        \item If $g_\pm\not\equiv0$ 
	        then there exist two sequence of eigenpairs 
	        $\{(u_n,\lambda_n^+(g))\}_{n\in\mathbb{N}}$ 
	        and $\{(v_n,\lambda_n^-(g))\}_{n\in\mathbb{N}}$ 
	        such that 
	        \[
	            \int_{\mathbb{R}^n} g(x)|u_n(x)|^p\,dx=1\quad \text{ and }\quad 
	            \int_{\mathbb{R}^n} g(x)|v_n(x)|^p \,dx=-1
	        \]
	        and
	    	\begin{align*}
	            &0<\lambda_1^+(g)<\lambda_2^+(g)\le\cdots\le\lambda_n^+(g)
	            \to\infty 
	            \,\text{  and  }\\
	            &0>\lambda_1^-(g)>\lambda_2^-(g)\ge\cdots\ge
	            \lambda_n^-(g)\to-\infty.
            \end{align*}
	        Moreover
	        \begin{align*}
	                \lambda_1^+(g)
	                &=\min\left\{\llbracket u\rrbracket_{s,p}^p\colon
	                u\in \wsprc \text{ and } 
	                \int_{\mathbb{R}^n} g(x)|u(x)|^p\,dx=1\right\},\\
	                -\lambda_1^-(g)&=\min\left\{\llbracket u\rrbracket_{s,p}^p
	                \colon 
	                 u\in \wsprc \text{ and }
	                 \int_{\mathbb{R}^n} g(x)|u(x)|^p\,dx=-1\right\},
            \end{align*}
            are simple eigenvalues with constant sign eigenfunctions.
        \end{enumerate}
    \end{theorem} 
    
    \medskip

	Let now discuss the case $sp\geq N$ and give the following result. 

	\begin{theorem}\label{teo:Esp>N}
        Assume $sp\ge N$ and $g=g_1-g_2$ satisfies 
        \begin{itemize}
	        \item $g_1(x)\ge0$ a.e. in $\mathbb{R}^N$ and 
	            $g_1\in\lir\cap L^{\nicefrac{N_0}{sp}}(\mathbb{R}^N)
	            \setminus\{0\},$
	            with $N_0\in\mathbb{N}$ such that $N_0>sp;$
	        \item $g_2(x)\ge\varepsilon>0$ a.e. in $\mathbb{R}^N.$ 
        \end{itemize} 
        Then there exists a sequence of eigenpairs 
	        $\{(u_n,\lambda_n)\}_{n\in\mathbb{N}}$ such that 
	        \[
	            \int_{\mathbb{R}^n} g(x)|u_n(x)|^p\,dx=1\quad\forall 
	            n\in\mathbb{N},
	        \]
	        and
	        \[
	            0<\lambda_1(g)<\lambda_2(g)\le\cdots\le\lambda_n(g)\to\infty.
	        \]
	        Moreover
	        \[
	            \lambda_1(g)=\min\left\{\llbracket u\rrbracket_{s,p}^p\colon 
	            u\in \wspr \text{ and }\int_{\mathbb{R}^n} 
	            g(x)|u(x)|^p\,dx=1\right\},
	          \]
	       is a simple eigenvalue with constant sign eigenfunction. 
	\end{theorem}

	\medskip
	
	We also give the following non-existence results.

	\begin{theorem}\label{theorem:nonexistence}
        If $sp>N,$ $g\in \lir$ and $   
        \displaystyle\int_{\mathbb{R}^N} g(x) dx >0$ then there is not a 
        positive principle eigenvalue.
    \end{theorem}

	The existence part, as in the local case, 
	is based on Lusternik--Schnirelman principle, 
	for details see \cite{MR0269962,MR578914,MR768749}.

  	\bigskip
  	
  	On the other hand, if we assume $g\in \lir$ then 
  	the solution found in the above theorems are H\"older continuous  
  	(see  Section \ref{HR}). 
  	Therefore we get 
  	$$
  		\lim\limits_{|x|\to+\infty}u(x)=0
  	$$
  	for any eigenfunction. Motivated by this result, we  study the 
  	asymptotic behaviour of positive eigenfunctions. 
  	More precisely, our fourth result
  	is a sharp decay estimate for the positive eigenfunction 
  	$u$ associated with $\lambda_1$ given by Theorem \ref{teo:Esp>N}. 
  	
  \begin{theorem}\label{theorem:decay1}
           Assume $sp\ge N$ and $g=g_1-g_2$ satisfies 
        \begin{itemize}
	        \item $g_1(x)\ge0$ a.e. in $\mathbb{R}^N$ and 
	            $g_1\in\lir\cap L^{\nicefrac{N_0}{sp}}
	            (\mathbb{R}^N)\setminus\{0\},$
	            with $N_0\in\mathbb{N}$ such that $N_0>sp;$
	        \item $g_2(x)\ge\varepsilon>0$ a.e. in $\mathbb{R}^N$ and 
	            $g_2\in\lir;$
	        \item $g(x)<-\delta<0$ for $|x|$ large enough.
        \end{itemize} 	
        Let $u\in\wspr$ be a positive eigenfunction associated to $\lambda_1(g).$ 
        Then
        there exists $k>>1$ such that
        \begin{equation}
        	\label{aaa}
        	C_1|x|^{-\frac{N+sp}{p-1}}\le u(x)\le 
            C_2|x|^{-\frac{N+sp}{p-1}}
        \end{equation}
        for any $|x|>k$ and some positive constants $C_1$ and $C_2$. 
    \end{theorem}
    
    The base in established this sharp decay estimate is Lemma \ref{lemma:decay},
    that is a nonlinear version of  \cite[Lemma 2.1]{MR3122168}(case $p=2$). 
    This lemma is the computation of the fractional $p-$Laplacian for a power 
    like function at infinity that give good sub and super-solutions. 
	Moreover, these sub and super-solutions can also be used to prove 
	decay estimate Schr\"odinger type equations, such result, in the case $p=2$, 
	can be found in \cite{MR3002595}.
	We remark that our nonlinear version of \cite[Lemma 2.1]{MR3122168}, 
	can be useful for other proposes like for example in the study of parabolic 
	problems as in \cite{MR3122168}.
	
	We would also like to remark that the above theorem 
	shows a difference between the local and nonlocal cases 
	since in the local case the eigenfunctions decay exponentially at 
    infinity, see \cite{MR1364493}. 

	\bigskip

	Finally, we are concerned with the decay rate at infinity of all 
	positive ground state solutions of the next autonomous Schr\"odinger equations
    \begin{equation}\label{eq:asc}
	    \begin{cases}
            (-\Delta_p)^su(x)+\mu|u|^{p-2}u=f(u) &\text{in }\mathbb{R}^N,\\
            u\in\wspr\\
            u(x)>0 &\text{for all }x\in\mathbb{R}^N, \mu>0.	    
	    \end{cases}
    \end{equation} 
 	
 	The existence of at least one positive ground state solution of 
 	\eqref{eq:asc} was recently proved in \cite{Ambrosio2018} under standard 
 	assumptions on $f$ including sub critical growth, 
 	for details see  $(f_1)-(f_5)$ in Section 
 	\ref{DE} and \cite{Ambrosio2018} 
 	where also other references for existence results can 
	be found.   
    
	\medskip

    In our last main results, we prove that the positive ground state solutions 
    of \eqref{eq:asc} also satisfies \eqref{aaa} for large $|x|$ large.
    
     \begin{theorem}\label{theorem:decay2}
	    Let $sp<N$ and suppose that $f$ verifies $(f_1)-(f_5).$
	    If $v$ is a positive ground state solutions of \eqref{eq:asc}
	    there is $k>>1$ such that 
	     \[
            C_1|x|^{-\frac{N+sp}{p-1}}\le v(x)\le 
            C_2|x|^{-\frac{N+sp}{p-1}}\quad \forall |x|>k
        \]
        for some positive constant $C_1$ and $C_2.$ 
    \end{theorem}

	\bigskip
	
	To end this introduction, we want to mention that, as far as we know, 
	the main results of this work are new also in the linear case $p=2$
	that corresponds to the fractional Laplacian.	
	
	\bigskip

	The rest of this paper is organized as follows: in Section \ref{NyP}, 
	we introduce the the notation a preview some preliminaries. 
	In Section \ref{sp<N} (resp, Section \ref{sp>N}), we study the case $sp<N$ 
	(resp. $sp\geq N$).  
	In Section \ref{HR}, we obtained the H\"older regularity.  
	In Section \ref{PE}, we study the principal eigenvalue. Finally, in Section 
	\ref{DE}, we prove the decay estimates.

\section{Notations and preliminaries}\label{NyP}
    For the benefit of the reader, we start by including the 
    basic tools    that will be needed in subsequent sections. 
    The known results are generally stated without proofs, 
    but we provide references where they can be found. 
    In additional, we take this opportunity to introduce 
    some of our notational conventions.
    
    \medskip
\subsection{Sobolev spaces}    
    Let $\Omega$ be an open subset of $N-$dimensional euclidean
    space $\mathbb{R}^N.$ Let $\cio$ denote the space 
    of infinitely differentiable functions on $\Omega;$ by 
    $\ccir$ we denote the space of functions in
    $\cir$ with compact support on $\Omega.$
    
    \medskip
    
    Let $1\le p\le\infty$ and $\lpo$ be the space of 
    Lebesgue measurable functions $u$ on $\Omega,$ such that
    \[
        \|u\|_{\lpo}\coloneqq
        \begin{cases}
	        \left(\displaystyle\int_{\Omega}
	        |u(x)|^p\,dx\right)^{\nicefrac1p}  
	        &\text{if } 1\le p<\infty,\\[9pt]
	        \inf \{ M \colon |u(x)| \le M  \text{ for almost every } x \} 
	        &\text{if } p=\infty,\\
        \end{cases}
    \] 
    is finite. If $\Omega=\mathbb{R}^N,$ we simply use the notation 
    $\|u\|_p$
    instead of $\|u\|_{\scriptstyle L^p(\mathbb{R}^N)}.$ 
    
    Let $0<s<1$ and $1< p<\infty.$ The fractional Sobolev spaces $\wspo$
	is defined to be the set of functions $u\in\lpo$ such that
		\[
			|u|_{\wspo}^p\coloneqq
			\int_{\Omega}\int_{\Omega}
			\dfrac{|u(x)-u(y)|^p}{|x-y|^{N+sp}}\, dxdy<\infty.
		\]
	The fractional Soblev spaces admit the following norm
		\[
			\|u\|_{\wspo}\coloneqq
			\left(\|u\|_{\lpo}^p+|u|_{\wspo}^p
			\right)^{\frac1p}.
		\]
	
	The space $\wspo$ endowed with the norm $\|\cdot\|_{\wspo}$ is 
	a reflexive Banach space.
    We denote by $\twspo$ the space of all $u\in\wspo$ such that
	$\tilde{u}\in\wspr,$ where $\tilde{u}$ is the extension by zero of 
	$u.$ 	
	\medskip
		
	The closure of $\ccir$ with respect to the
		norm
		\[
			\llbracket u\rrbracket_{s,p}^p\coloneqq
			\int_{\mathbb{R}^N}\int_{\mathbb{R}^N}
			\dfrac{|u(x)-u(y)|^p}{|x-y|^{N+sp}}\, dxdy
		\]
	is denoted by $\wsprc.$
	
	\medskip
	
	In the case $sp<N,$ we consider the spaces
	\[
	    \mathcal{X}_0\!\prescript{s,p}{}(\Omega)\coloneqq
        \left\{u\in\lps\colon u\equiv 0\text{ in } 
        \mathbb{R}^N\setminus\Omega, \int_{\mathbb{R}^N}\int_{\mathbb{R}^N}
	    \dfrac{|u(x)-u(y)|^p}{|x-y|^{N+sp}}\, dxdy<\infty\right\},
	\]
	and 
	\[
	    \mathcal{X}\!\prescript{s,p}{}(\mathbb{R}^N)\coloneqq
        \left\{u\in\lps\colon \int_{\mathbb{R}^N}\int_{\mathbb{R}^N}
			\dfrac{|u(x)-u(y)|^p}{|x-y|^{N+sp}}\, dxdy<\infty\right\},
	\]
	 where $p_s^*=\dfrac{pN}{N-sp}.
	 $
    \medskip
    
    The proof of the following results can be found in 
    \cite[page 521]{Mazya}.   

    \begin{theorem}\label{theorem:Mazya}
        If $sp<N,$ then for an arbitrary
        function $u\in\wsprc$ there holds
        \[
            \|u\|_{p_s^*}^p\le C(N,p)\dfrac{s(1-s)}{(N-sp)^{p-1}}\llbracket 
            u\rrbracket_{s,p}^p
        \]
        where $C(N,p)$ is a function of 
        $N$ and $p.$
    \end{theorem}
    
    In other words, if $sp<N$ then
    $\wsprc\subseteq\mathcal{X}\!\prescript{s,p}{}(\mathbb{R}^N).$
    In fact, adapting ideas of the proofs  of Proposition 4.27 in 
    \cite{DD} and 
    Proposition 2.4 in \cite{MR1101219} we can proof the following result.

    \begin{theorem}\label{theorm:c}
        If $sp<N$ then $\wsprc=\mathcal{X}\!\prescript{s,p}{}(\mathbb{R}^N).$
    \end{theorem}
    
    \begin{proof}
	    By Theorem \ref{theorem:Mazya}, we only need to show that 
	    $\mathcal{X}\!\prescript{s,p}{}(\mathbb{R}^N)\subseteq\wsprc.$ 
	    
	    Let $u\in \mathcal{X}\!\prescript{s,p}{}
	    (\mathbb{R}^N)$ and $\varphi\in C^\infty_0(\mathbb{R}^N)$
	    be such that
	    \begin{equation}\label{eq:CroNig}
	        0\le\varphi\le1 \text{ in }\mathbb{R}^N,  \varphi(x)=1 
	        \text{ if } |x|\le1
	        \text{ and } \varphi(x)=0 \text{ if } |x|\ge2.
        \end{equation}
	    We define $\varphi_n(x)\coloneqq\varphi(\nicefrac{x}{n})$ and 
	    $u_n(x)=\varphi_n(x)u(x).$ 
	    
	   \medskip
	    
	   \noindent{\it Step 1.} We claim that $u_n\in \wspr.$
	   
	   Since $\varphi$ has compact support and $u\in\lps,$ we have that $u_n\in\lpr.$
	   Therefore we only need to show that 
	   \begin{equation}\label{eq:ArgIsl}
	        \int_{\mathbb{R}^N}\int_{\mathbb{R}^N}
			\dfrac{|u_n(x)-u_n(y)|^p}{|x-y|^{N+sp}}\, dxdy<\infty.
       \end{equation}
	   
	   Observe that
	   \begin{align*}
	        &\int_{\mathbb{R}^N}\int_{\mathbb{R}^N}
			\dfrac{|u_n(x)-u_n(y)|^p}{|x-y|^{N+sp}}\, dxdy\\
			&\le C\left(
	        \int_{\mathbb{R}^N}\int_{\mathbb{R}^N}
			\dfrac{|u(x)-u(y)|^p}{|x-y|^{N+sp}}\, dxdy
			+\int_{\mathbb{R}^N}\int_{\mathbb{R}^N}
			\dfrac{|\varphi_n(x)-\varphi_n(y)|^p |u(x)|^p}{|x-y|^{N+sp}}\, dxdy\right).
       \end{align*}
       Then, to prove \eqref{eq:ArgIsl} it suffices to show that
       the second term on the right-hand sides of the above inequality 
       is finite.
       
       Let $B_n=\{x\in\mathbb{R}^N\colon |x|<n\}$ and 
       $B_n^c=\mathbb{R}^N\setminus B_n.$
       Then
       \begin{align*}
	        &\int_{\mathbb{R}^N}\int_{\mathbb{R}^N}
			\dfrac{|\varphi_n(x)-\varphi_n(y)|^p |u(x)|^p}{|x-y|^{N+sp}}
			\, dxdy=\\
			&\int_{B_n}\int_{B_n^c}
			\dfrac{|\varphi_n(x)-1|^p |u(x)|^p}{|x-y|^{N+sp}}\, dxdy+
			\int_{B_n^c}\int_{\mathbb{R}^N}
			\dfrac{|\varphi_n(x)-\varphi_n(y)|^p |u(x)|^p}{|x-y|^{N+sp}}
			\, dxdy\\
			&=I_1+I_2.
       \end{align*}
       By \eqref{eq:CroNig} and H\"older's inequality, we get
       \begin{align*}
	        I_1\le&C
	        \left\{\int_{B_n^c\cap B_{2n}}\dfrac{|\varphi_n(x)-1|^p |u(x)|^p}{(|x|-n)^{sp}}\, 
	        dx\right.\\
	        &\quad\left.+\|u\|_{p_s^*}^p\int_{B_n}
	         \left(\int_{B_{2n}^c}
	        \dfrac{dx}{|x-y|^{N+\nicefrac{N^2}{sp}}}\right)^{\nicefrac{sp}N}dy
	        \right\}\\
	        \le& C\|u\|_{p_s^*}^p\left\{
	        \left(\int_{B_n^c\cap B_{2n}}
	        \dfrac{|\varphi_n(x)-\varphi_n(\nicefrac{nx}{|x|})|
	        ^{\nicefrac{N}s}}{(|x|-n)^{N}}\, dx\right)^{\nicefrac{sp}N}
	        +1
	        \right\}\\
	        \le& C\|u\|_{p_s^*}^p\left\{
	        \left(\int_{B_n^c\cap B_{2n}}
	        \|\nabla \varphi_n\|_\infty^{\nicefrac{N}s-N}
	        (|x|-n)^{\nicefrac{N}s-N}\, dx\right)^{\nicefrac{sp}N}
	        +1
	        \right\}\\
	        <&\infty.
       \end{align*}
        
       On the other hand, again by \eqref{eq:CroNig} and H\"older's inequality,
       \begin{align*}
	        &I_2=\int_{B_n^c}\int_{B_{2n}}
			\dfrac{|\varphi_n(x)-\varphi_n(y)|^p |u(x)|^p}{|x-y|^{N+sp}}\, dxdy
			+\int_{B_n^c}\int_{B_{2n}^c}
			\dfrac{|\varphi_n(y)|^p |u(x)|^p}{|x-y|^{N+sp}}\, dxdy\\
	        &\le\int_{B_{2n}}\int_{B_{2n}}
			\dfrac{|\varphi_n(x)-\varphi_n(y)|^p |u(x)|^p}{|x-y|^{N+sp}}\, dxdy
			+\int_{B_{2n}^c}\int_{B_{2n}}
			\dfrac{|\varphi_n(x)|^p |u(x)|^p}{|x-y|^{N+sp}}\, dxdy\\
			&\qquad+\int_{B_{2n}}\int_{B_{2n}^c}
			\dfrac{|\varphi_n(y)|^p |u(x)|^p}{|x-y|^{N+sp}}\, dxdy\\
			&\le C\left\{\int_{B_{2n}}\int_{B_{2n}}
			|x-y|^{p(1-s)-N} |u(x)|^p\, dxdy +\int_{B_{2n}}
			\dfrac{|\varphi_n(x)|^p |u(x)|^p}{(2n-|x|)^{sp}}dx\right\}\\
			&+\|u\|_{p_s^*}^p\int_{B_{2n}}|\varphi_n(y)|^p\left(\int_{B_{2n}^c}
			\dfrac{dx}{|x-y|^{N+\nicefrac{N^2}{sp}}}\right)^{\nicefrac{sp}{N}}dy
			\\
			&\le C\left\{\|u\|_{p_s^*}^p+\int_{B_{2n}}
			(2n-|x|)^{p(1-s)} |u(x)|^p dx
			+\int_{B_{2n}}\dfrac{|\varphi_n(y)|^p}{(2n-|y|)^N} dy
			\|u\|_{p_s^*}^p\right\}\\
			&\le C\left\{\|u\|_{p_s^*}^p
			+\int_{B_{2n}}(2n-|y|)^{p-N} dy
			\|u\|_{p_s^*}^p\right\}\\
			&<\infty.
       \end{align*}       
       Therefore $I_1+I_2<\infty$ then \eqref{eq:ArgIsl} holds.
       
       \medskip
        
       \noindent{\it Step 2.} We now claim that 
       $\llbracket u_n-u\rrbracket_{s,p}\to 0.$
       
	   It is clear that $u_n\to u$ strongly in $\lps.$ 
	   Our first step is to prove that  $\llbracket u_n -u\rrbracket_{s,p}\to0$ 
	   as $n\to\infty.$
	   Observe that
	   \begin{equation*}
	     \begin{aligned}
	            \llbracket u_n -u\rrbracket_{s,p}^p&= 2\int_{B_n}\int_{B_n^c} 
	            \dfrac{v_n(x,y)}{|x-y|^{N+sp}}dydx
	                +\int_{B_n^c}\int_{B_n^c}
	                \dfrac{v_n(x,y)}{|x-y|^{N+sp}}dydx\\
	                &= 2J_n^1+J_n^2,
            \end{aligned}   
        \end{equation*}
	    where $v_n(x,y)=|(u_n(x)-u(x))-(u_n(y)-u(y))|^p.$ 
	    Then, we only need to show that
	    $J_n^1$ and $J_n^2$ converge to $0.$
	    
	    Let's prove that $J_n^1\to 0$ as $n\to\infty.$ 
	    By \eqref{eq:CroNig} and H\"older's inequality
	    \begin{align*}
	        J_n^1&\le\int_{B_{n}^c\cap B_{2n}}|1-\varphi_n(y)|^p|u(y)|^p\int_{B_n}
	        \dfrac{dxdy}{|x-y|^{N+sp}}
	        +\int_{B_n}\int_{B_{2n}^c} 
	            \dfrac{|u(y)|^p}{|x-y|^{N+sp}}dydx
	        \\
	        &\le C(N,s,p)\int_{B_{n}^c\cap B_{2n}}
	        \dfrac{|\varphi_n(\nicefrac{ny}{|y|})-\varphi_n(y)|^p|u(y)|^p}{(|y|-n)^{sp}}dy\\
	         &\qquad+\int_{B_n}\left(\int_{B_{2n}^c} 
	            \dfrac{dy}{|x-y|^{N+\nicefrac{N^2}{sp}}}\right)^{\nicefrac{sp}N}dx
	            \left(\int_{B_n^c}
	        |u(y)|^{p_s^*}dy\right)^{\frac{N-sp}N}\\
	        &\le \dfrac{C(N,s,p,\varphi)}{n^p}\int_{B_{n}^c\cap B_{2n}}
	        (|y|-n)^{p(1-s)}|u(y)|^pdy\\
	         &\qquad+ C(N,s,p)\int_{B_n}\dfrac{dx}{(2n-|x|)^N}
	            \left(\int_{B_n^c}
	        |u(y)|^{p_s^*}dy\right)^{\frac{N-sp}N}\\
	        &\le \dfrac{C(N,s,p,\varphi)}{n^p}\left(\int_{B_{2n}}
	        (|y|-n)^{\frac{N}s(1-s)}dy\right)^{\nicefrac{sp}N}\left(\int_{B_n^c}
	        |u(y)|^{p_s^*}dy\right)^{\frac{N-sp}N}\\
	         &\qquad+ C(N,s,p)
	            \left(\int_{B_n^c}
	        |u(y)|^{p_s^*}dy\right)^{\frac{N-sp}N},
        \end{align*}
        by a simple change of variable 
        \begin{align*}
	        J_n^1&\le C(N,s,p,\varphi)\left(\int_{B_{2}}
	        (|y|-1)^{\frac{N}s(1-s)}dy\right)^{\nicefrac{sp}N}\left(\int_{B_n^c}
	        |u(y)|^{p_s^*}dy\right)^{\frac{N-sp}N}\\
	         &\qquad+ C(N,s,p)
	            \left(\int_{B_n^c}
	        |u(y)|^{p_s^*}dy\right)^{\frac{N-sp}N}\\
	        &\le C(N,s,p,\varphi)\left(\int_{B_n^c}
	        |u(y)|^{p_s^*}dy\right)^{\frac{N-sp}N}\to 0 \text{ since } u\in \lps.
        \end{align*}
        
        Our next aim is to show that $J_n^2\to 0$ as $n\to\infty.$ Observe that for any
        $x,y\in B_n^c$ we have
        \[
            u_n(x)-u_n(y)=
            \begin{cases}
                \varphi_n(x)u_n(x) &\text{if }x\in B_n^c\cap B_{2n}, y\in B_{2n}^c,\\
               -\varphi_n(y)u_n(y) &\text{if }y\in B_n^c\cap B_{2n}, x\in B_{2n}^c,\\
               \varphi_n(x)u_n(x)-\varphi_n(x)u_n(y)&\text{if }x,y\in B_n^c\cap B_{2n},\\
               0&\text{if }x,y\in  B_{2n}^c.
            \end{cases}
        \]
        
        Then
        \begin{align*}
            J_n^2\le &C(p)\left\{\int_{B_{n}^c}\int_{B_{n}^c}	
            \dfrac{|u(x)-u(y)|^p}{|x-y|^{N+sp}}\, dxdy
            +2\int_{B_n^c\cap B_{2n}}\int_{B_{2n}^c}
            \dfrac{|\varphi_n(y)u(y)|^p}{|x-y|^{N+sp}}
            \, dxdy\right.\\
            &\int_{B_n^c\cap B_{2n}}\int_{B_n^c\cap B_{2n}}
            |\varphi_n(x)|^p\dfrac{|u(x)-u(y)|^p}{|x-y|^{N+sp}}\, dxdy\\
            &+\left.\int_{B_n^c\cap B_{2n}}\int_{B_n^c\cap B_{2n}}
             \dfrac{|\varphi_n(x)-\varphi_n(y)|^p|u(y)|^p}{|x-y|^{N+sp}}\, 
             dxdy\right\}\\
            &\le C(p)\left\{2\int_{B_{n}^c}\int_{B_{n}^c}	
            \dfrac{|u(x)-u(y)|^p}{|x-y|^{N+sp}}\, dxdy
            +2\int_{B_n^c\cap B_{2n}}\int_{B_{2n}^c}
            \dfrac{|\varphi_n(y)u(y)|^p}{|x-y|^{N+sp}}
            \, dxdy\right.\\
            &+\left.\int_{B_n^c\cap B_{2n}}\int_{B_n^c\cap B_{2n}}
             \dfrac{|\varphi_n(x)-\varphi_n(y)|^p|u(y)|^p}{|x-y|^{N+sp}}\, 
             dxdy\right\}. 
        \end{align*}
        Since $u\in \mathcal{X}\!\prescript{s,p}{}
	    (\mathbb{R}^N),$ by dominated convergence theorem, we get
	    \[
	        \int_{B_{n}^c}\int_{B_{n}^c}	
            \dfrac{|u(x)-u(y)|^p}{|x-y|^{N+sp}}\, 
            dxdy\to0\quad \text{as }n\to\infty.
	    \]
	    
	    On the other hand, by  H\"older's inequality
        \begin{align*}
	        \int_{B_n^c\cap B_{2n}}&\int_{B_{2n}^c}
	        \dfrac{|\varphi_n(y)u(y)|^p}{|x-y|^{N+sp}}
	        \, dxdy=C(N,s,p)\int_{B_n^c\cap B_{2n}}
	        \dfrac{|\varphi_n(y)u(y)|^p}{|2n-|y||^{sp}}
	       \, dy\\
	       \le&\dfrac{C(N,s,p,\varphi)}{n^p}\int_{B_n^c\cap B_{2n}}
	       |2n-|y||^{(1-s)p}|u(y)|^pdy \\
	       \le&C(N,s,p,\varphi)\left(\int_{B_n^c}
	        |u(y)|^{p_s^*}dy\right)^{\frac{N-sp}N}\to 0 \text{ since } u\in \lps.
        \end{align*}
        
        Finally
        \begin{align*}
	        \int_{B_n^c\cap B_{2n}}\int_{B_n^c\cap B_{2n}}&
             \dfrac{|\varphi_n(x)-\varphi_n(y)|^p|u(y)|^p}{|x-y|^{N+sp}}\, dxdy\\
             &\le \dfrac{C(\varphi)}{n^p}\int_{B_n^c\cap B_{2n}}\int_{B_n^c\cap B_{2n}}
             |x-y|^{p(1-s)-N}|u(y)|^p dxdy\\
             &\le \dfrac{C(N,s,p,\varphi)}{n^{sp}}\int_{B_n^c\cap B_{2n}}|u(y)|^p dy\\
             &\le C(N,s,p,\varphi)\left(\int_{B_n^c}
	        |u(y)|^{p_s^*}dy\right)^{\frac{N-sp}N}\to 0 \text{ since } u\in \lps.
        \end{align*}
        
        Hence, $J_n^2\to 0$ as $n \to \infty.$
        
        \noindent{\it Step 3.} Finally, we show that $u\in \wsprc.$
        
        By step 1, we have that $\{u_n\}_{n\in\mathbb{N}}\subset\wspr.$ Then for all 
        $n\in \mathbb{N}$ there is $\phi_n\in\ccir$ such that 
        $\llbracket u_n-\phi_n\rrbracket_{s,p}\le \nicefrac{1}n.$ Therefore
        \[
            \llbracket u-\phi_n\rrbracket_{s,p}\le
            \llbracket u-u_n\rrbracket_{s,p}+\dfrac1{n}\to 0 
            \text{ as } n\to \infty \text{ (step 2).}
        \]
        Thus $u\in\wsprc.$
    \end{proof}
    
    It is easy to see that  $\wsprc$ is a reflexive banach space.
    Now the proof of the following results is standard.
    
    \medskip

	\begin{corollary}
		Let  $sp<N,$ and $g\in \lpsn.$ Then there is a 
		positive constant $C$ such that
		\[
			\int_{\mathbb{R}^N}g(x)|u(x)|^p\, dx\le C
			\llbracket u\rrbracket_{s,p}^p		
		\]
		for all $u\in\wsprc.$
	\end{corollary} 
	
	\begin{corollary}\label{coro:importante}
		Let  $sp<N,$ and $g\in \lpsn.$  If $\{u_n\}_{n\in\mathbb{N}}$ 
		is a sequence of 
		$\wsprc$ such that $u_{n}\rightharpoonup u$ weakly in $\wsprc,$
		then there is a subsequence $\{u_{n_k}\}_{k\in\mathbb{N}}$ such that
		\[
	        \begin{aligned}
	            \int_{\mathbb{R}^N}g(x)|u_{n_k}(x)|^p dx&\to 
	            \int_{\mathbb{R}^N}g(x)|u(x)|^p dx,\\
	            \int_{\mathbb{R}^N}g(x)|u_{n_k}(x)|^{p-2}u_{n_k}(x)u(x) dx&\to 
	            \int_{\mathbb{R}^N}g(x)|u(x)|^p dx.
            \end{aligned}
        \]
	\end{corollary}     
	
	For more details about these spaces and their use, 
    we refer the reader to 
    \cite{Adams, DD, DNPV, Grisvard,MR3593528, Mazya}.	
\subsection{The principal eigenvalue in bounded domain.}       
    
    We start by introducing the definition of eigenpair in bounded domain.
    
     \begin{definition}\label{def:eigend} 
        Let $\Omega\subset\mathbb{R}^N$ be a bounded 
        domain with Lipschitz   
        boundary, $s\in(0,1)$ $p\in(1,\infty)$ and 
        $g\in \mathcal{A}(\Omega)\coloneqq\{f\in L^{\infty}(\Omega)
        \colon |\{x\in\Omega\colon f(x)>0\}|>0\}.$
        We say that a pair $(u,\mu)\in\twspo\times\mathbb{R}$ 
        is a weak solution of
        \begin{equation}\label{eq:epd}
	           \begin{cases}
                    (-\Delta_p)^s u(x)=\mu g(x)|u(x)|^{p-2}u(x)
	                    & \text{ in }\Omega,\\
	                    u(x)=0 &  \text{ in }
	                    \mathbb{R}^N\setminus\Omega,\\
	            \end{cases}
        \end{equation}
        if
        \begin{equation}\label{eq:wsd}
	          \begin{aligned}
                   \int_{\mathbb{R}^N}\int_{\mathbb{R}^N}
                     &\dfrac{|u(x)-u(y)|^{p-2}(u(x)-u(y))(\varphi(x)-\varphi(y))}
                     {|x-y|^{N+ps}}\,dxdy\\
                   &\hspace{4cm}=\mu\int_{\mathbb{R}^N}
                   g(x)|u(x)|^{p-2}u(x)\varphi(x) dx
              \end{aligned}
         \end{equation}              
         for all $\varphi\in C_0^\infty(\Omega).$    
         A pair $(u,\mu)$ is called a Dirichlet eigenpair if 
         $u$ is nontrivial 
         and $(u,\mu)$ is a weak solution of \eqref{eq:epd}.
         In which case, $\mu$ is called an Dirichlet eigenvalue and  $u$ a
         corresponding eigenfunction.                 
    \end{definition}
    
     Given $g\in\mathcal{A}(\Omega),$ the first Dirichlet eigenvalue is
    \begin{equation}
	    \label{eq:RQ}
	        \mu_1(\Omega,g) \coloneqq\min\left\{\llbracket u\rrbracket_{s,p}^p\colon
	            u\in\twspo\text{ and } 
	            \int_{\mathbb{R}^n} g(x)|u(x)|^p\,dx=1\right\}.
     \end{equation}
     Moreover, we have the following result.
     
     \begin{theorem}\label{teo:LL}
        Let $\Omega\subset\mathbb{R}^N$ be a bounded domain with Lipschitz 
        boundary, $s\in(0,1),$ $p\in(1,\infty),$ and 
        $g\in\mathcal{A}.$  There exists a positive 
        function $u\in\twspo\cap\lir,$ such that
        \begin{itemize}
		    \item $u$ is a minimizer of \eqref{eq:RQ};
		    \item $(u,\mu_1(\Omega,g))$ is a weak solution of 
		    \eqref{eq:epd}.
        \end{itemize}
        Furthermore $\mu_1(\Omega,g)$ is simple.
    \end{theorem}
     \begin{proof}
	    See \cite[Section 4]{MR3556755}.
    \end{proof}  

    Finally by a scaling argument, we have the following result.    
    \begin{lemma}\label{lemma:auxeigen}
        Let $B_R$ be the ball of center $0$ and radius $R$ and 
            $\mu_R=\mu_1(B_R,1).$  Then
            \[
                 \mu_1(B_R,1)\le\dfrac1{R^{sp}}\mu_1(B_1,1).
            \] 
    \end{lemma}    

\section{Case $sp<N$}\label{sp<N}
  As mentioned in the introduction  we shall establish the existence of a sequence of eigenvalues using the 
    Lusternik--Schnirelman principle, see \cite{MR0269962,MR578914,MR768749}.
    
    \medskip
    
    Let us consider
    \[
        \mathcal{M}
        \coloneqq
        \left\{u\in\wsprc\colon p\Psi(u)\coloneqq\int_{\mathbb{R}^N} 
        g(x)|u(x)|^p dx=1 \right\},
    \]
    and $\Phi\colon \wsprc\to\mathbb{R}$
    \[
        \Phi(u)\coloneqq\dfrac1p\int_{\mathbb{R}^N}\int_{\mathbb{R}^N}
			\dfrac{|u(x)-u(y)|^p}{|x-y|^{N+sp}}\, dxdy.
    \]
    It is known that $\Phi$ is weakly lower semicontinuous and 
    that $\Phi$ and $\Psi$ are of class $C^1.$
    
    \medskip
    
    Observe that,  $(u,\lambda)\in \mathcal{M}\times\mathbb{R}$ 
    is an eigenpair if only if u is a critical point of 
    $\Phi$ restricted to the manifold $\mathcal{M}.$ 
    Then, we are looking for the critical points of $\Phi$ 
    restricted to the manifold $\mathcal{M}.$ To find them, we will use the  
    Lusternik--Schnirelman 
    principle. For this reason we need to show Palais-Smale condition for the 
    functional $\Phi$ on $\mathcal{M}.$
    
    \begin{lemma}\label{lemma:plc}
	    Assume that $sp<N,$ $g\in \lir\cap\lpsn$ and $g\not\equiv0.$ 
	    Then the functional $\Phi$ satisfies the Palais--Smale
	    condition on $\mathcal{M}$.
    \end{lemma}
    
    \begin{proof}
	    Given a sequence $\{u_n\}_{n\in\mathbb{N}}\subset \mathcal{M}$ such that 
	    $\{\Phi(u_n)\}_{n\in\mathbb{N}}$ is bounded and 
	    \begin{equation}\label{eq:aps}
	            \Phi'(u_n)-\llbracket u_n\rrbracket_{s,p}^p
	            \Psi'(u_n)\to 0\quad\text{as }n\to\infty,
        \end{equation}
	    we want to show that there exist a function $u\in \wsprc$ 
	    and a subsequence $\{u_{n_k}\}_{k\in\mathbb{N}}$ 
	    of $\{u_n\}_{n\in\mathbb{N}}$ such that 
	    \[
	            u_{n_k}\to u \quad\text{strongly in }\wsprc. 
	    \]
	    
	    Since $\wsprc\subset\wspo$ for any bounded smooth domain $\Omega$
	    and  $\{\Phi(u_n)\}_{n\in\mathbb{N}}$ 
	    is bounded, there exist a function $u\in \wsprc$ 
	    and a subsequence $\{u_{n_k}\}_{k\in\mathbb{N}}$ of $\{u_n\}_{n\in\mathbb{N}}$ 
	    such that 
	    \begin{equation}\label{eq:aps1}
	        \begin{aligned}
	           u_{n_k}&\rightharpoonup u \quad\text{weakly in }\wsprc,\\
	         u_{n_k}&\to u \quad\text{strongly in }\lpo,
            \end{aligned}
        \end{equation}
	    for any bounded smooth domain $\Omega.$ Furthermore, 
	    since $g\in\lpsn,$ by Corollary \ref{coro:importante} we have
	    that
	    \begin{equation}\label{eq:aps2}
	        \begin{aligned}
	            \int_{\mathbb{R}^N}g(x)|u_{n_k}(x)|^p dx&\to 
	            \int_{\mathbb{R}^N}g(x)|u(x)|^p dx,\\
	            \int_{\mathbb{R}^N}g(x)|u_{n_k}(x)|^{p-2}u_{n_k}(x)u(x) dx&\to 
	            \int_{\mathbb{R}^N}g(x)|u(x)|^p dx.
            \end{aligned}
        \end{equation}
        Therefore $u\in \mathcal{M}.$  
        
        On the other hand, by \eqref{eq:aps} we get
        \begin{align*}
	        \llbracket u_{n_k}&\rrbracket_{s,p}^{p-1} \llbracket u\rrbracket_{s,p}\ge\\&\ge
	        \int_{\mathbb{R}^N}\int_{\mathbb{R}^N}
                    \dfrac{|u_{n_k}(x)-u_{n_k}(y)|^{p-2}(u_{n_k}(x)-u_{n_k}(y))
                    (u(x)-u(y))}
                    {|x-y|^{N+ps}}\,dxdy\\
                    &=\llbracket u_{n_k}\rrbracket_{s,p}^p\int_{\mathbb{R}^N} g(x) 
                    |u_{n_k}(x)|^{p-2}u_{n_k}(x)u(x) dx
                    +o(1).
        \end{align*}
        Then, by \eqref{eq:aps1} and \eqref{eq:aps2}, we have
        \[
            \llbracket u\rrbracket_{s,p}\ge\liminf_{k\to\infty}\llbracket u_{n_k}
            \rrbracket_{s,p}\ge \llbracket u\rrbracket_{s,p},
        \]
        that is
        \[
            \liminf_{k\to\infty}\llbracket u_{n_k}\rrbracket_{s,p}=\llbracket u\rrbracket_{s,p}.
        \]
	    Hence, taking a subsequence if necessary, $u_{n_k}\to u$ strongly in $\wsprc.$   
	\end{proof}
	
	 Thus, by the Lusternik--Schnirelman principle, 
	 we get (i) of Theorem \ref{thm:existence1}. 
	 For (ii)  we consider additionally 
    \[
        \mathcal{M}^-
        \coloneqq
        \left\{u\in\wsprc\colon p\Psi(u)\coloneqq\int_{\mathbb{R}^N} 
        g(x)|u(x)|^p dx=-1 \right\},
    \]
that is not empty by assumption and so again by the Lusternik-Schnirelman principle we get the second sequence. The principal eigenvalue results of Theorem \ref{thm:existence1} are discuss in Section  \ref{PE}.
  
\section{Cases $sp\ge N$}\label{sp>N}
\subsection{Nonexistence result} 
    
    Our next aim is to show nonexistent result Theorem \ref{theorem:nonexistence} . The next result will be one of the keys to prove our nonexistence result.
    
     \begin{lemma}\label{lemma:nonexistence}
        If $sp>N,$ $g\in \lir$ and $   
        \displaystyle\int_{\mathbb{R}^N} g(x) dx >0$ then 
        \[
            \mu_1(B_R,g)\to 0 \text { as } R\to \infty.
        \]
        Here $B_R$ denotes the ball of center $0$ and radius $R.$
    \end{lemma}
    \begin{proof}
        Let $\varphi\in C^{\infty}_0(B_1)$ be  such that $\varphi(0)=1.$ For all $R>1$ we 
        define 
        \[
            \varphi_R(x)\coloneqq \varphi\left(\dfrac{x}{R}\right)\in 
            C^{\infty}_0(B_R). 
        \]
        
        If $\displaystyle\int_{\mathbb{R}^N} g(x) dx<\infty,$ by dominated convergence
        theorem
        \[
             \int_{\mathbb{R}^N} g(x)|\varphi_R(x)|^p dx 
             \to \int_{\mathbb{R}^N} g(x) dx
        \] 
        as $R\to\infty.$
        
        On the other hand, if $\displaystyle\int_{\mathbb{R}^N} g(x) dx=\infty,$ 
        by dominated convergence theorem
        \[
	        \int_{\mathbb{R}^N} g_{-}(x)|\varphi_R(x)|^p dx \to
	         \int_{\mathbb{R}^N} g_-(x) dx 
        \]
        as $R\to\infty,$ and by Fatou's lemma
        \[
            \liminf_{R\to\infty}\int_{\mathbb{R}^N} g_{+}(x)|\varphi_R(x)|^p dx
            \ge  \int_{\mathbb{R}^N} g_+(x) dx=\infty. 
        \]
        Therefore, in both cases, 
        \begin{equation}\label{eq:aux1nonexist}
            \int_{\mathbb{R}^N} g(x)|\varphi_R(x)|^p dx \to \int_{\mathbb{R}^N} g(x) dx
        \end{equation}
	    as $R\to\infty.$
	    
	    Since $\displaystyle\int_{\mathbb{R}^N} g(x) dx>0,$
	    there exists $R_0>1$ such that 
	    \[
	            g\in\mathcal{A}(B_R)\quad\text{ and }
	            \quad \int_{\mathbb{R}^N} g(x)|\varphi_R(x)|^p dx>0 \quad\forall R >R_0.
        \]
	    Then,  
	    \begin{align*}
	        0&<\mu_1(B_R,g)\\&\le\dfrac{\llbracket \varphi_R\rrbracket_{s,p}^p}
	        {\displaystyle\int_{\mathbb{R}^N} g(x)|\varphi_R(x)|^p dx}=
	        \dfrac{1}{R^{sp-N}}\dfrac{\llbracket \varphi\rrbracket_{s,p}^p}
	        {\displaystyle\int_{\mathbb{R}^N} g(x)|\varphi_R(x)|^p dx}\to 0 \text{ as } R
	        \to \infty
        \end{align*}
	  because $sp>N.$
    \end{proof}
    
    We now prove our non-existence result.
    
    \begin{proof}[Proof of Theorem \ref{theorem:nonexistence}]
        Assume by contradiction the existence of a weak solution $(\lambda,u)$
        of $\eqref{eq:ep}$ such that $\lambda>0$ and $u>0.$ 
        
        Since $\displaystyle \int_{\mathbb{R}^N} g(x) dx>0$  there exists $R_0>1$
        such that $g\in\mathcal{A}(B_R)$ for all $R>R_0.$ 
        Then, by Theorem \ref{teo:LL}, 
        for any $R>R_0$ there exists a positive function
        $v_R\in\widetilde{W}\!\prescript{s,p}{}(B_R)\cap L^{\infty}(\mathbb{R}^N)$ 
        such that $(v_R,\mu_1(B_R,g))$ is a weak solution of \eqref{eq:epd} 
        with $\Omega=B_R$ and 
        \[
            \int_{\Omega}g(x)v_R(x)^p dx=1.
        \]
        
        Let $n\in\mathbb{N}$ and $u_n(x)\coloneqq u(x)+\nicefrac1n.$
        Since $v_R, u\in\wspr\cap L^\infty(\mathbb{R}^N)$ and $v_R=0$ in 
        $\mathbb{R}^N\setminus B_R,$ we have that 
        \[
            w_m(x)=\dfrac{v_R(x)^p}{u_m(x)^{p-1}}\in \widetilde{W}\!\prescript{s,p}{}(B_R).
        \] 
        For further details we refer the reader to \cite[Proof of Theorem 4.8]{MR3556755}.
        
        By the discrete version of Picone's identity (see \cite{Amghibech}), we have
        \begin{align*}
    	    0\le&
    	     \int_{\mathbb{R}^N}\int_{\mathbb{R}^N}
    	        \dfrac{|v_R(y)-v_R(x)|^p}{|x-y|^{n+sp}} dx dy\\ 
    	        &-\int_{\mathbb{R}^N}\int_{\mathbb{R}^N}
    	        \dfrac{|u(y)-u(x)|^{p-2}(u(y)-u(x))}{|x-y|^{n+sp}}
		        \left(\dfrac{v_R(y)^p}{u_m(y)^{p-1}}-\dfrac{v_R(x)^p}{u_m(x)^{p-1}}
		        \right)dxdy\\
		        \le&
		        \mu_{1}(B_R,g)-
		            \lambda\int_{B_R}g(x)u(x)^{p-1}
		            \dfrac{v_R(x)^p}{u_m(x)^{p-1}}\, dx.
	    \end{align*}
        Then, by dominated convergence theorem, we have
        \[
            0\le \mu_{1}(B_R,g) -\lambda
        \]
        that is
        \[
            \lambda\le\mu_1(B_R,g)\quad\forall R>R_0.
        \]
	    Therefore $\lambda=0$ since $\mu_1(B_R,g)\to 0$ as $R\to\infty$ 
	    by Lemma \ref{lemma:nonexistence}. This contradiction establishes the result.
    \end{proof}

\subsection{Existence result}

    It follows from the proof of Theorem \ref{thm:existence1} 
    that to show the existence of a sequence  
    of eigenvalues we need to prove the following:   
    if $\{u_n\}_{n\in\mathbb{N}}$ is contained in
    \[
       \mathcal{N}
       \coloneqq \left\{u\in\wspr\colon \displaystyle
       \int_{\mathbb{R}^N} g(x)|u|^p dx=1\right\}
    \]  
    and $\{\Phi(u_n)\}_{n\in\mathbb{N}}$ is bounded then  $\{u_n\}_{n\in\mathbb{N}}$
    is bounded in $\wspr.$ To prove this we will adapt ideas of \cite{MR1356326} 
    for the nonlocal case.
    
    \medskip
    We now want to prove existence result Theorem \ref{teo:Esp>N}. 
    The key in the proof of Theorem \ref{teo:Esp>N} is the following result.
    
    \begin{theorem}\label{theorm:embedding}
        Let $f\in\lir\cap L^{\nicefrac{N_0}{sp}}(\mathbb{R}^N)\setminus\{0\},$
	    for some integer $N_0>\max\{N,sp\},$ $M=N_0-N,$ and
	    $B_R^{M}$ be the ball of center $0$ and radius $R$ in $\mathbb{R}^M.$ 
	    Then there is a constant $C=C(N,M,s,p)$ 
	    such that
	     \[
	         \int_{\mathbb{R}^N} |f(x)||\varphi(x)|^p\,dx\le 
	        C R^{\nicefrac{Msp}{N_0}}\|f\|_{\nicefrac{N_0}{sp}}
	        \left(\llbracket \varphi \rrbracket_{s,p}^p+
	            \mu_1(B_R^M,1)
	            \|\varphi\|_p^p\right)
        \]
        for in $\varphi\in\ccir.$
    \end{theorem}
    \begin{proof}
        Let 
        $(v,\mu_1(B_R^M,1))$ be a Dirichlet eigenpair such that $v>0$ in $B_R^M$ 
        and
        \[
            \int_{\mathbb{R}^M}|v(x)|^p\, dx =1.
        \]

        \medskip
        
        Given
        $\varphi \in \ccir,$ as in the proof of \cite[Lemma 3]{MR1356326},  we define 
        \[
            h,\tau\colon\mathbb{R}^N\times\mathbb{R}^M\to \mathbb{R}
        \]
        \[
            h(x,y)\coloneqq f(x)\cchi_{B_R^M}(y)\quad \text{ and } \quad
            \tau(x,y)\coloneqq\varphi(x)v(y),\quad (x,y)\in \mathbb{R}^N\times\mathbb{R}^M
        \]
        where $\cchi_{B_R^M}(y)$ is the characteristic function $B_R^M.$
        
        Observe that 
        \begin{equation}\label{eq:eta0}
	        \int_{\mathbb{R}^N} |f(x)||\varphi(x)|^p\,dx= 
	        \iint_{\mathbb{R}^N\times\mathbb{R}^M} |h(x,y)||\tau(x,y)|^p dx dy.
       \end{equation}   
      
   \noindent{\it Claim.} 
        $\tau\in W^{s,p}(\mathbb{R}^N\times\mathbb{R}^M).$ Moreover
         \begin{equation}\label{eq:eta1}
	        \begin{aligned}
	            &\iint_{\mathbb{R}^N\times\mathbb{R}^M}
	            \iint_{\mathbb{R}^N\times\mathbb{R}^M} 
	            \dfrac{|\tau(x,y)-\tau(w,z)|^p}{(|x-w|^2+|y-z|^2)^{\frac{M+N+sp}2}} 
	            dxdydwdz\le\\
	            &\le C\left(\int_{\mathbb{R}^N}\!\!\int_{\mathbb{R}^N}
	            \dfrac{|\varphi(x)-\varphi(w)|^p}{|x-w|^{N+sp}} dxdw+
	            \mu_1(B_R^M,1)
	            \int_{\mathbb{R}^N} |\varphi(w)|^p dw\right),
            \end{aligned}
          \end{equation}   
          with the constant $C$ depending only on $N,M,s$ and $p.$
          
          Note that
            \[
                \iint_{\mathbb{R}^N\times\mathbb{R}^M} |\tau(x,y)|^p dxdy=
                \left(\int_{\mathbb{R}^N} |\varphi(x)|^p dx\right)
                \left( \int_{\mathbb{R}^M} |v(y)|^p dy\right).
            \]
            Therefore $\tau\in L^p (\mathbb{R}^N\times\mathbb{R}^M).$
            
            On the other hand
            \begin{equation}\label{eq:eta2}
	            \begin{aligned}
	            &\iint_{\mathbb{R}^N\times\mathbb{R}^M}
	            \iint_{\mathbb{R}^N\times\mathbb{R}^M} 
	            \dfrac{|\tau(x,y)-\tau(w,z)|^p}{(|x-w|^2+|y-z|^2)^{\frac{M+N+sp}2}} 
	            dxdydwdz\le\\
	            &\le C(p)\left(\iint_{\mathbb{R}^N\times\mathbb{R}^M}
	            \iint_{\mathbb{R}^N\times\mathbb{R}^M} 
	            \dfrac{|\tau(x,y)-\tau(w,y)|^p}{(|x-w|^2+|y-z|^2)^{\frac{M+N+sp}2}} 
	            dxdydwdz\right.\\
	            &+\left. \iint_{\mathbb{R}^N\times\mathbb{R}^M}
	           \iint_{\mathbb{R}^N\times\mathbb{R}^M} 
	            \dfrac{|\tau(w,y)-\tau(w,z)|^p}{(|x-w|^2+|y-z|^2)^{\frac{M+N+sp}2}} 
	            dxdydwdz\right).
            \end{aligned}
        \end{equation}   
            
        For the first term on the right hand side of the previous inequality  we have
             \begin{equation}\label{eq:eta3}
	            \begin{aligned}
	            &\iint_{\mathbb{R}^N\times\mathbb{R}^M}
	            \iint_{\mathbb{R}^N\times\mathbb{R}^M} 
	            \dfrac{|\tau(x,y)-\tau(w,y)|^p}{(|x-w|^2+|y-z|^2)^{\frac{M+N+sp}2}}  
	            dxdydwdz=\\
	            &=\int_{\mathbb{R}^N}\!\!\int_{\mathbb{R}^N}
	            |\varphi(x)-\varphi(w)|^p
	            \int_{\mathbb{R}^M} |v(y)|^p\int_{\mathbb{R}^M}  
	            \dfrac{dzdydxdw}{(|x-w|^2+|y-z|^2)^{\frac{M+N+sp}2}}\\
	            &=C(N,M) \int_{\mathbb{R}^N}\!\!\int_{\mathbb{R}^N}
	            \dfrac{|\varphi(x)-\varphi(w)|^p}{|x-w|^{N+sp}} dxdw
	           \int_{0}^\infty  \!\!
	            \dfrac{r^{M-1}dr}{(1+r^2)^{\frac{M+N+sp}2}}\\ 
	            &=C(N,M,s,p) \int_{\mathbb{R}^N}\int_{\mathbb{R}^N}
	            \dfrac{|\varphi(x)-\varphi(w)|^p}{|x-w|^{N+sp}} dxdw.
	       \end{aligned}
        \end{equation}   
	        
	        Similarly
	         \begin{equation}\label{eq:eta4}
	         \begin{aligned}
	            &\iint_{\mathbb{R}^N\times\mathbb{R}^M}
	            \iint_{\mathbb{R}^N\times\mathbb{R}^M} 
	            \dfrac{|\tau(w,y)-\tau(w,z)|^p}{(|x-w|^2+|y-z|^2)^{\frac{M+N+sp}2}}  
	            dxdydwdz=\\
	            &=C(N,M,s,p) \int_{\mathbb{R}^N}\int_{\mathbb{R}^N}
	            \dfrac{|v(y)-v(z)|^p}{|y-z|^{N+sp}} dydz
	            \int_{\mathbb{R}^M} |\varphi(w)|^p dw\\
	            &=C(N,M,s,p) \mu_1(B_R^M,1)
	            \int_{\mathbb{R}^M} |\varphi(w)|^p dw.
	        \end{aligned}
        \end{equation}   
        Therefore, by \eqref{eq:eta2}, \eqref{eq:eta3} and \eqref{eq:eta4}, we get
        \eqref{eq:eta1}.
        
        \medskip
        
        Since $N_0=N+M>sp$ and $\tau\in W^{s,p}(\mathbb{R}^N\times\mathbb{R}^M),$ 
        we have that $|\tau|^p\in L^{\nicefrac{N_0}{(N_0-sp)}}
        (\mathbb{R}^N\times\mathbb{R}^M).$ Then by \eqref{eq:eta0}, H\"older's inequality,
        Theorem \ref{theorem:Mazya} and \eqref{eq:eta1}, we get
        \[
	         \int_{\mathbb{R}^N} |f(x)||\varphi(x)|^p\,dx\le 
	        C R^{\nicefrac{Msp}{N_0}}\|f\|_{\nicefrac{N_0}{sp}}
	        \left(\llbracket \varphi \rrbracket_{s,p}^p+
	            \mu_1(B_R^M,1)
	            \|\varphi\|_p^p\right)
        \]
        with the constant $C$ depending only on $N,M,s$ and $p.$
    \end{proof}
    
   \begin{corollary}\label{corollary:embedding}
	  Assume $sp\ge N$ and $g=g_1-g_2$ satisfies 
        \begin{itemize}
	        \item $g_1(x)\ge0$ a.e. in $\mathbb{R}^N$ and 
	            $g_1\in\lir\cap L^{\nicefrac{N_0}{sp}}(\mathbb{R}^N)\setminus\{0\},$
	            with $N_0\in\mathbb{N}$ such that $N_0>sp;$
	        \item $g_2(x)\ge\varepsilon>0$ a.e. in $\mathbb{R}^N.$ 
        \end{itemize} 
        Then there is a constant $C$ such that 
        \[
            \|u\|_p^p\le C(1+\llbracket u\rrbracket_{s,p}^p).
        \]
        for all $u\in \mathcal{N}.$
    \end{corollary}
    \begin{proof}
	    Let $u\in\mathcal{N}.$ Then by Theorem \ref{theorm:embedding}
	    and Lemma \ref{lemma:auxeigen} we get
	    \begin{align*}
	        \varepsilon\int_{\mathbb{R}^N}|u(x)|^pdx&\le
	        \int_{\mathbb{R}^N}g_2(x)|u(x)|^pdx=1+\int_{\mathbb{R}^N}g_1(x)|u(x)|^pdx\\
	        &\le 1+ C R^{\nicefrac{Msp}{N_0}}\|g_1\|_{\nicefrac{N_0}{sp}}
	        \left(\llbracket u \rrbracket_{s,p}^p+
	            \mu_1(B_R^M,1)
	            \|u\|_p^p\right)\\
	         &\le 1+ C R^{\nicefrac{(N_0-N)sp}{N_0}}\|g_1\|_{\nicefrac{N_0}{sp}}
	        \left(\llbracket u \rrbracket_{s,p}^p+
	            \dfrac1{R^{sp}}\mu_1(B_1,1)
	            \|u\|_p^p\right)
        \end{align*} 
        where $C$ is a constant independent of $u$ and $R.$ Then for $R$ large enough, 
        we have that there is a constant $C$ independent of $u$ such that
        \[
            \|u\|_p^p\le C(1+\llbracket u\rrbracket_{s,p}^p).
        \]
    \end{proof}
   
     Using Corollary \ref{corollary:embedding} and proceeding as 
     in the proof of Lemma \ref{lemma:plc} we can prove that  
     the functional $\Phi$ satisfies the Palis--Smale
	 condition on $\mathcal{N}.$
    
    \begin{lemma}\label{lemma:plc2}
	    Let $s\in(0,1)$ and  $p\in(1,\infty)$ be such that $sp\ge N.$ 
        Assume $g=g_1-g_2$ satisfies 
        \begin{itemize}
	        \item $g_1(x)\ge0$ a.e. in $\mathbb{R}^N$ and 
	            $g_1\in\lir\cap L^{\nicefrac{N_0}{sp}}(\mathbb{R}^N)\setminus\{0\},$
	            with $N_0\in\mathbb{N}$ such that $N_0>sp;$
	        \item $g_2(x)\ge\varepsilon>0$ a.e. in $\mathbb{R}^N.$ 
        \end{itemize} 
	    Then the functional $\Phi$ satisfies the Palais--Smale
	    condition on $\mathcal{N}.$
    \end{lemma}
    
    Finally, by the Lusternik--Schnirelman principle, Theorem \ref{teo:Esp>N} follows.
     
    \section{H\"older regularity}\label{HR}
    
    In this section, we will show that if $(u, \lambda)$ is an eigenvalue then
    $u\in \lir\cap C^\gamma(\mathbb{R}^N)$ for some $\gamma\in(0,1).$ 
    
    \medskip
    
    By the fractional Sobolev embedding theorem, we know that if $sp>N$ then
    $\wspr\subset \lir\cap C^{s-\nicefrac{N}p}(\mathbb{R}^N).$ Then,
    we have the following result.
    
    \begin{lemma}\label{lemma:regularidad1}
	     If $sp>N$ and $(u,\lambda)$ is an eigenpair then
	     $u\in \lir\cap C^{s-\nicefrac{N}p}(\mathbb{R}^N).$ 
    \end{lemma}
    
    \medskip
    
    In the case $sp\le N,$ we need assume that $g\in\lir$ to proof the result.

     \begin{lemma}\label{leama:linfnorm}
	    Assume $sp\le N$ and $g\in \lir.$ 
	    If $(u,\lambda)$ is an eigenpair then $u\in\lir.$
     \end{lemma}
     
     \begin{proof} We split the proof in two cases.
        
        \medskip
       
	    \noindent {\it Case $sp<N.$}
	    We begin by observing that $(-u,\lambda)$ is an eigenpair 
	    since $(u,\lambda)$ is.
	    Therefore, it is enough to prove that $u_+\in\lir.$
	    
	    Observe that $u_+\in\wsprc$ because 
	    \[
	        |u_+(x)-u_+(y)|\le|u(x)-u(y)| \quad\forall x,y\in\mathbb{R}^N.
	    \]

        We now intend to prove by induction on $n$  that 
  	    \[
  	        u_+\in L^{p\gamma_n}(\mathbb{R}^N),\quad \gamma_n=\left(\frac{N}{N-sp}\right)^n
  	    \]
	    for all $n\in\mathbb{N},$ with the estimate
	    \begin{equation}\label{eq:alinfnorm}
	          \|u_+\|_{\gamma_{n+1}p}
           \le K^{\nicefrac1{\gamma_n}}\gamma_n^{\nicefrac1{\gamma_n}}\|u_+\|_{\gamma_np},
        \end{equation}
	    for some positive constant $K=K(N,s,p,\lambda,\|g\|_{\lir}).$
	    
	    Since $u_+\in\wsprc,$ we have that $u\in L^{p\gamma_1}(\mathbb{R}^N).$
	    Assume now $u_+\in L^{p\gamma_n}(\mathbb{R}^N)$ for some positive integer $n.$
	    We want to show that $u_+\in L^{p\gamma_{n+1}}(\mathbb{R}^N)$.
	    
	    Let us  define, for any positive integer $k$
	    \[
	        v_k(x)=\min\{k,u_+(x)\}.
	    \]  
	    Since $u_+\in\wsprc$ and for any $\alpha>1$
	    \begin{align*}
	         |(v_k(x))^\alpha-(v_k(y))^\alpha|
	        &\le \alpha (v_k(x)+v_k(y))^{\alpha-1}  |v_k(x)-v_k(y)|\\
	        &\le \alpha (2k)^{\alpha-1}  |u_+(x)-u_+(y)| \quad\forall x,y\in\mathbb{R}^N.
        \end{align*}
	    we get $v_k^\alpha \in \lir\cap\wsprc$ for any $\alpha>1.$ 
	    
	    On the other hand, if $u_+(x)\ge u_+(y)$ then $v_k(x)\ge v_k(y)$ and 
	    there is $\theta\in(v_k(y),v_k(x))$ such that
	    \begin{align*}
	         |(v_k(x))^{\gamma_n}&-(v_k(y))^{\gamma_n}|^p
	        \le {\gamma_n}^p \theta^{(\gamma_n-1)p}  |v_k(x)-v_k(y)|^p\\
	        &\le {\gamma_n}^p \theta^{(\gamma_n-1)p}   
	        |u_+(x)-u_+(y)|^{p-2}(u_+(x)-u_+(y))(v_k(x)-v_k(y))\\
	        &\le {\gamma_n}^p \theta^{(\gamma_n-1)p}   
	        |u_+(x)-u_+(y)|^{p-2}(u_+(x)-u_+(y))v_k(x)\\
	        &- {\gamma_n}^p \theta^{(\gamma_n-1)p}   
	        |u_+(x)-u_+(y)|^{p-2}(u_+(x)-u_+(y))v_k(y)\\
	         &\le {\gamma_n}^p 
	        |u_+(x)-u_+(y)|^{p-2}(u_+(x)-u_+(y))v_k(x)^{(\gamma_n-1)p+1}\\
	        &- {\gamma_n}^p 
	        |u_+(x)-u_+(y)|^{p-2}(u_+(x)-u_+(y))v_k(y)^{(\gamma_n-1)p+1}.
        \end{align*}
	    Then, taking $\alpha_n=(\gamma_n-1)p+1,$ we get 
	    \begin{align*}
	         |(v_k(x))^{\gamma_n}&-(v_k(y))^{\gamma_n}|^p\le\\
	        &\le|u_+(x)-u_+(y)|^{p-2}(u_+(x)-u_+(y))[(v_k(x))^{\alpha_n}
	        -(v_k(y))^{\alpha_n}]\\
	        &\le|u(x)-u(y)|^{p-2}(u(x)-u(y))[(v_k(x))^{\alpha_n}
	        -(v_k(y))^{\alpha_n}].
        \end{align*}
        By symmetry we also get the last inequality when $u_+(y)\ge u_+(x).$ 
	    Thus, by Theorem \ref{theorem:Mazya}, there is a positive constant $C=C(N,s,p)$
	    such that
	    \begin{align*}
	        &C^p\|v_k^{\gamma_n}\|_{p_s^*}^p\le \llbracket v_k^{\gamma_n}
	        \rrbracket_{s,p}^p\\
	        &\le \gamma_n^p
	        \int_{\mathbb{R}^N}\int_{\mathbb{R}^N}
	        \dfrac{|u(x)-u(y)|^{p-2}(u(x)-u(y))[(v_k(x))^{\alpha_n}
	        -(v_k(y))^{\alpha_n}]}{|x-y|^{N+sp}}\, dxdy.
        \end{align*}
         Therefore, since $(u,\lambda)$ is an eigenpair, we get
        \[
           C^p\|v_k\|_{\gamma_{n+1}p}^{\gamma_np}= C^p\|v_k^{\gamma_n}\|_{p_s^*}^p
           \le  |\lambda|\gamma_n^p\|g\|_{\lir}
            \|u_+\|_{\gamma_np}^{\gamma_np}.
        \]
        Then there is a positive constant $K=K(N,s,p,\lambda,\|g\|_{\lir})$ 
        such that
        \[
            \|v_k\|_{\gamma_{n+1}p}
           \le K^{\nicefrac1{\gamma_n}}\gamma_n^{\nicefrac1{\gamma_n}}\|u_+\|_{\gamma_np}.
        \]
        Letting $k\to\infty,$ by Fatou's Lemma
        \[
            \|u_+\|_{\gamma_{n+1}p}
           \le K^{\nicefrac1{\gamma_n}}\gamma_n^{\nicefrac1{\gamma_n}}\|u_+\|_{\gamma_np}.
        \]
        Hence $u_+\in L^{\gamma_{n+1}p}(\mathbb{R}^N)$ and 
        \begin{align*}
	        \|u_+\|_{\gamma_{n+1}p}&\le\prod_{i=1}^n K^{\nicefrac1{\gamma_i}}
	        \gamma_i^{\nicefrac1{\gamma_i}}\|u_+\|_{\gamma_1p}= 
	        K^{\displaystyle\sum_{i=1}^{n}\nicefrac1{\gamma_1^i}}
	        \gamma_1^{\displaystyle\sum_{i=1}^{i}\nicefrac{i}{\gamma_1^i}}
	        \|u_+\|_{\gamma_1p}\\
	        &=K^{\frac{\gamma_1^n-1}{(\gamma_1-1)\gamma_1^n}}
	        \gamma_1^{\frac{\gamma_1(\gamma_1^n-1)-n(\gamma_1-1)}{(\gamma_1-1)^2\gamma_1^n}}
	        \|u_+\|_{\gamma_1p}.
        \end{align*}
        Note that 
        \[
            \lim_{n\to \infty} K^{\frac{\gamma_1^n-1}{(\gamma_1-1)\gamma_1^n}}=
            K^{\nicefrac1{(\gamma_1-1)}}\quad\text{and}\quad
            \lim_{n\to \infty}
            \gamma_1^{\frac{\gamma_1(\gamma_1^n-1)-n(\gamma_1-1)}{(\gamma_1-1)^2\gamma_1^n}}
            =\gamma_1^{\nicefrac{\gamma_1}{(\gamma_1-1)^2}}.
        \]
        
        Now, it is easy to check that  $u_+\in \lir.$
        
        \medskip
        
        \noindent {\it Case $sp=N.$} Let $\tilde{s}\in(0,s)$ be such that $\tilde{s}p<N.$ 
        By 
        \cite[Lemma 2.3]{MR3556755} there is a constant $K=K(N,\tilde{s},s,p)$ such that
        \[
            \llbracket v \rrbracket_{\tilde{s},p}^p\le\llbracket v 
            \rrbracket_{\tilde{s},p}^p
            + K\|v\|_p^p \quad\forall v\in\wspr.
        \]
        Thus, by Theorem \ref{theorem:Mazya}, we get that there is a constant 
        $C=C(N,\tilde{s},p)$ such that
        \[ 
            C^p\|v\|_{p_{\tilde{s}}^*}^p
            \le\llbracket v \rrbracket_{\tilde{s},p}^p\le\llbracket v 
            \rrbracket_{\tilde{s},p}^p
            + K\|v\|_p^p \quad\forall v\in\wspr
        \]
        that is
        \[ 
            C^p\|v\|_{p_{\tilde{s}}^*}^p
           \le\llbracket v 
            \rrbracket_{\tilde{s},p}^p
            + K\|v\|_p^p \quad\forall v\in\wspr.
        \]
        
        Now, taking $\gamma_n=\left(\frac{N}{N-\tilde{s}p}\right)^n$
         and proceeding as in the previous case, we can conclude that $u\in\lir.$
    \end{proof}

    Then, by the previous lemma and \cite[Corollary 5.5]{MR3593528} 
    we have the following result.
   
    \begin{lemma}\label{lemma:regularidad2}
	     Assume $sp\le N$ and $g\in \lir.$ If $(u,\lambda)$ is an eigenpair then
	     $u\in \lir\cap C^{\gamma}(\mathbb{R}^N)$ for some $\gamma\in(0,1).$ 
    \end{lemma}
\section{Principal eigenvalues}\label{PE}
    Now we will collect some relevant properties  of the principal 
    eigenvalues and their eigenfunctions. 
    To simplify matters, in the remainder of this section we write
    $\lambda_1=\lambda_1(g)$ ($\lambda_1^\pm=\lambda_1^\pm(g)$). 
     
	\medskip
	
	Our next lemma follows from the following inequality
	\[
	    \llbracket|u|\rrbracket_{s,p}<\llbracket u\rrbracket_{sp}
	\] 
    for any $u$ such that $u_\pm\not\equiv0.$
      
	\begin{lemma}\label{lemma:PositiveFE1} 
		If $u$ is an eigenfunction associated to $\lambda_1$ ($\lambda_1^\pm$) 
	    then either $u_+\equiv0$ or $u_-\equiv0.$ 
	\end{lemma}
    
    Moreover, by \cite[Theorems 1.2 and 1.4]{MR3631323}, we get the following results.
    
    \begin{lemma}\label{lemma:PositiveFE2} 
        Assume that $sp<N$ and $g\in \lir\cap\lpsn.$
        If $g_+\not\equiv0$ ($g_\pm\not\equiv0$) and $u$ is an eigenfunction 
        associated to $\lambda_1$ ($\lambda_1^\pm$) then either 
        $u>0$ or $u<0$  in whole $\mathbb{R}^N.$

    \end{lemma}      
    
    \begin{lemma}\label{lemma:PositiveFE3}
	    Let $s\in(0,1)$ and  $p\in(1,\infty)$ be such that $sp\ge N.$ 
        Assume $g=g_1-g_2$ satisfies 
        \begin{itemize}
	        \item $g_1(x)\ge0$ a.e. in $\mathbb{R}^N$ and 
	            $g_1\in\lir\cap L^{\nicefrac{N_0}{sp}}(\mathbb{R}^N)\setminus\{0\},$
	            with $N_0\in\mathbb{N}$ such that $N_0>sp;$
	        \item $g_2(x)\ge\varepsilon>0$ a.e. in $\mathbb{R}^N.$ 
        \end{itemize} 
        If $u$ is an eigenfunction
        associated to $\lambda_1$ then either 
        $u>0$ or $u<0$ a.e. in $\mathbb{R}^N.$
        
        If in addition $g_{2}\in C(\mathbb{R}^N)$ or $g_{2}\in \lir$ then either 
        $u>0$ or $u<0$ in whole $\mathbb{R}^N.$    
	    
    \end{lemma}

    The proof of the result given below follows from a careful reading of 
    \cite[proof of Theorem 4.8]{MR3556755}
    
     \begin{lemma}\label{lemma:Simple} 
        If $u$ is an eigenfunction associated to
        $\lambda_1$ ($\lambda_1^\pm$) such that $u>0$ a.e. in $\mathbb{R}^N$ 
	    and $\lambda\ge \lambda_1$ ($\lambda\ge \pm\lambda_1^\pm$) 
	    is such that there exists a nonnegative eigenfunction 
	    $v$ associated to $\lambda$ ($\pm\lambda$)
	    then $\lambda=\lambda_1$ ($\lambda=\pm\lambda_1^\pm$) 
	    and there is $k\in\mathbb{R}$ such that $v = ku_1$
	    ($v = ku_1^\pm$) in $\mathbb{R}^N.$ 
	 \end{lemma}    
    
    Then, by Lemmas \ref{lemma:PositiveFE1}, \ref{lemma:PositiveFE2}, 
    \ref{lemma:PositiveFE3} 
    and \ref{lemma:Simple}, we have the next two theorems that give the last part of our main theorems.

    \begin{theorem}\label{theorm:PrincipalAndSimple1}
        If  $sp<N,$  $g\in \lir\cap\lpsn,$
        $g_+\not\equiv0$ ($g_\pm\not\equiv0$) 
        then $\lambda_1$ ($\lambda_1^\pm$) is simple 
        and all its eigenfunctions are of constant sign.                 
    \end{theorem}
    
     \begin{theorem}\label{theorm:PrincipalAndSimple2}
        If $sp\ge N$  and $g=g_1-g_2$ satisfies 
        \begin{itemize}
	        \item $g_1(x)\ge0$ a.e. in $\mathbb{R}^N$ and 
	            $g_1\in\lir\cap L^{\nicefrac{N_0}{sp}}(\mathbb{R}^N)\setminus\{0\},$
	            with $N_0\in\mathbb{N}$ such that $N_0>sp;$
	        \item $g_2(x)\ge\varepsilon>0$ a.e. in $\mathbb{R}^N,$ 
        \end{itemize} 
        then $\lambda_1$  is simple 
        and all its eigenfunctions are of constant sign.                 
    \end{theorem}

\section{Decay estimates}\label{DE} 
   Finally,  we study the decay rate at infinity of
   \begin{itemize}
	    \item all positive eigenfunctions associated to $\lambda_1(g)$ 
	        in the case $sp\ge N;$
	    \item all positive ground state solutions of the autonomous 
	        Schr\"odinger equations in the case $sp< N.$ 
   \end{itemize}   
   For this reason, we give an nonlinear version of \cite[Lemma 2.1]{MR3122168}.
   
   \begin{lemma}\label{lemma:decay}
        Let $\Upsilon\in C^2(\mathbb{R}^N)$ be a positive function such that 
        $\Upsilon$ is radially symmetric and decreasing. 
        Assume also
        that 
        \begin{equation}\label{eq:H}
        	  \Upsilon(x)\le\dfrac{C_1}{|x|^\alpha},\quad 0\neq|\nabla 
        	 \Upsilon(x)|\le 
            \dfrac{C_2}{|x|^{\alpha+1}}, \quad |D^2 \varphi(x)|\le 
            \dfrac{C_3}{|x|^{\alpha+2}}
        \end{equation}
        for some $\alpha>0$ and for $|x|$ large enough. Then there is
        $k>>1$ such that
        \begin{equation}\label{eq:I}
	        |(-\Delta_p)^s\Upsilon(x)|\le
            \begin{cases}
	            \dfrac{c_1}{|x|^{\alpha(p-1)+ps}}&\text{ if } 
	            \alpha(p-1)<N,\\[10pt]
	            \dfrac{c_2\log(|x|)}{|x|^{N+ps}}&\text{ if } 
	            \alpha(p-1)=N,\\[10pt]
	            \dfrac{c_3}{|x|^{N+ps}}&\text{ if } \alpha(p-1)>N,
            \end{cases}
        \end{equation} 
        for all $|x|\ge k.$ Here $c_1,c_2,c_3$ are positive 
        constants
        that depend only on $\alpha,s,p, N$ and 
        $\|\Upsilon\|_{C^2(\mathbb{R}^N)}.$
        
        Moreover if $\alpha(p-1)>N$ and $|x|>k,$ we have
        \begin{equation}\label{eq:Ip}
        	(-\Delta_p)^s\Upsilon(x)\le-\dfrac{c_4}{|x|^{N+sp}}.
        \end{equation}
        for some positive constant $c_4.$
    \end{lemma}
   
   \begin{proof}
	    By \cite[Lemma 3.6]{KORVENPAA2017}, we have that  
	    $(-\Delta_p)^s\Upsilon(x)$
	    is finite for all $|x|$ large enough.
	    
	    Now we proceed as in the proof of Lemma 2.1 in \cite{MR3122168}.
	    From now on $|x|$ is large enough. 
	    \begin{equation}\label{eq:I0}
	        \begin{aligned}
	            (-\Delta_p)^s\Upsilon(x)
	            &=\int_{\mathbb{R}^N}\dfrac{|\Upsilon(x)-\Upsilon(y)|^{p-2}
	            (\Upsilon(x)-\Upsilon(y))}{|x-y|^{N+ps}}dy\\
	            &=\int_{|y|>\frac{3}{2}|x|} 
	            \dfrac{|\Upsilon(x)-\Upsilon(y)|^{p-2}
	            (\Upsilon(x)-\Upsilon(y))}{|x-y|^{N+ps}}dy\\
	            &+\int_{\left\{\frac{|x|}{2}\le|y|\le\frac{3}{2}|x|\right\}
	            \setminus B_{\frac{|x|}2}(x)}
	            \dfrac{|\Upsilon(x)-\Upsilon(y)|^{p-2}
	            (\Upsilon(x)-\Upsilon(y))}{|x-y|^{N+ps}}dy\\
	            &+\int_{B_{\frac{|x|}2}(x)}
	            \dfrac{|\Upsilon(x)-\Upsilon(y)|^{p-2}
	            (\Upsilon(x)-\Upsilon(y))}{|x-y|^{N+ps}}dy\\
	            &+\int_{B_{\frac{|x|}2}(0)}
	            \dfrac{|\Upsilon(x)-\Upsilon(y)|^{p-2}
	            (\Upsilon(x)-\Upsilon(y))}{|x-y|^{N+ps}}dy\\
	            &=I_1+I_2+I_3+I_4.
            \end{aligned}
        \end{equation}

	    If $|y|>\dfrac32|x|$ then $\Upsilon(y)<\Upsilon(x)$ and therefore
	    \begin{equation}\label{eq:I1}
	        \begin{aligned}
	           0\le I_1&\le |\Upsilon(x)|^{p-1}\int_{|y|>\frac{3}{2}|x|} 
	           \dfrac{dy}{|x-y|^{N+sp}}\\
	           &=C(N,s,p)\dfrac{|\Upsilon(x)|^{p-1}}{|x|^{ps}}
	           \le \dfrac{C(N,s,p,C_1,\alpha)}{|x|^{\alpha(p-1)+ps}}.
	        \end{aligned}
        \end{equation}
        
        In the same way
	    \begin{equation}\label{eq:I2}
	        \begin{aligned}
	            |I_2|&\le \left|\Upsilon\left(\dfrac{x}2\right)\right|^{p-1}
	           \int_{\left\{\frac{|x|}{2}\le|y|\le\frac{3}{2}|x|\right\}
	            \setminus B_{\frac{|x|}2}(x)}\dfrac{1}{|x-y|^{N+ps}}dy
	           \le \dfrac{C(N,s,p,C_1,\alpha)}{|x|^{\alpha(p-1)+ps}}.
	        \end{aligned}
        \end{equation}
        
        On the other hand, if $|y|\le\dfrac{|x|}2$ 
        then $\Upsilon(x)\le\Upsilon(y)$
        and $|x-y|>\dfrac{|x|}2.$ Therefore
			\begin{align*}
	            |I_4|&\le \dfrac{2^{N+ps}}{|x|^{N+ps}}
	            \int_{B_{\frac{|x|}2}(0)}|\Upsilon(y)|^{p-1}dy\\
	            &\le \dfrac{C(N,s,p)}{|x|^{N+ps}}
	            \int_{B_{\frac{|x|}2}(0)}|\Upsilon(y)|^{p-1}dy
	            =\dfrac{C(N,s,p)}{|x|^{N+ps}}
	            I_4'.
	        \end{align*}
		As in \cite{MR3122168}, we have the following estimate for $|x|$ large 
	    enough
	    \begin{equation}\label{eq:I4}
	        \begin{aligned}
	            &\bullet\text{If } \alpha(p-1)>N \text{ then } 
	            I_4'<\infty \text{ and }
	            |I_4|\le\dfrac{\Gamma_1}{|x|^{N+sp}};\\
	          &\bullet\text{If } \alpha(p-1)<N \text{ then } 
	          I_4'\text{ grows like }
	            |x|^{N-\alpha(p-1)} \text{ and }
	            |I_4|\le\dfrac{\Gamma_2}{|x|^{\alpha(p-1)+sp}};\\
	          &\bullet \text{Finally, if } \alpha(p-1)=N \text{ we get }
	          |I_4|\le  \Gamma_3\dfrac{\log(|x|)}{|x|^{N+ps}}.
            \end{aligned}	
        \end{equation}
	    Here $\Gamma_1,\Gamma_2,\Gamma_3$ are positive constant
        that depend only on $N,s,p,C_1,\alpha$ and 
        $\|\Upsilon\|_{L^\infty(\mathbb{R}^N)}.$
        
        \medskip
        
        The real difference with the linear case is observed in the estimate
        of $|I_3|$. It follows from the proof of Lemma 3.6 in 
        \cite{KORVENPAA2017} that there is a positive constant $c$
        depending on $N$ and $p$ such that
       \begin{equation}\label{eq:AI3}
      	 \dfrac{|I_3|}c\le\begin{cases}
      			\tau
      			\omega^{p-2}
      			\left(\dfrac{|x|}{2}\right)^{p(1-s)}+
      			\tau^{p-1}
      			\left(\dfrac{|x|}{2}\right)^{p(1-s)}&\text{if } p\ge 2,
      			\\[10pt]
      			\tau^{p-1}\left(\dfrac{|x|}{2}\right)^{p-2+p(1-s)}
      			&\text{if } \dfrac{2}{2-s}< p< 2,\\[10pt]
      			\tau\omega^{p-2}\left(\dfrac{|x|}{2}\right)^{p(1-s)}
      			&\text{if } 1<p\le\dfrac{2}{2-s},\\
      		\end{cases}
       \end{equation}
       where $\tau=\displaystyle\sup\left\{|D^2\Upsilon(y)|\colon
       y\in B_{\frac{|x|}2}(x)\right\}$ and $\omega=\displaystyle
       \sup\left\{|\nabla\Upsilon(y)|\colon
       y\in B_{\frac{|x|}2}(x)\right\}.$
       Since $\dfrac{|x|}{2}\le|y|$ for any $y\in B_{\frac{|x|}2}(x),$
       by \eqref{eq:H} and \eqref{eq:AI3}, we have
       \begin{equation}\label{eq:I3}
       		|I_3|\le \dfrac{C}{|x|^{\alpha(p-1)+N}}
       \end{equation}
      where $C$ is a positive constant depending on $N,s,p,C_2,C_3$
      and $\alpha.$
      
      By \eqref{eq:I1}, \eqref{eq:I2}, \eqref{eq:I4} and \eqref{eq:I3}
      we get \eqref{eq:I}.
      
      \medskip
      
      To end the proof we show \eqref{eq:Ip}
      For $x$ large enough,by \eqref{eq:I1}, \eqref{eq:I2}, and \eqref{eq:I3}
      there is a positive constant $C$ such that 
      \begin{equation}\label{eq:Ip1}
      		(-\Delta_p)^s\Upsilon(x)
      		=I_1+I_2+I_3+I_4\le I_4+\dfrac{C}{|x|^{\alpha(p-1)+ps}}.
      \end{equation}
      
      On the other hand, since $|y|\le\dfrac{|x|}2$ implies that
         $\Upsilon(x)\le\Upsilon(y)$ and $|x-y|<\dfrac32|x|,$ we get
      \begin{align*}
			I_4&=-
			\int_{B_{\frac{|x|}{2}}(0)}
			\dfrac{(\Upsilon(y)-\Upsilon(x))^{p-1}}{|x-y|^{N+ps}}dy\\
			&\le -\left(\dfrac{2}3\right)^{N+ps}\dfrac1{|x|^{N+ps}} 
			\int_{B_{\frac{|x|}{2}}(0)}
			(\Upsilon(y)-\Upsilon(x))^{p-1}dy.   
      \end{align*}
      Now, if $|x|>2$ we have that
      \begin{align*}
			I_4
			&\le -\left(\dfrac{2}3\right)^{N+ps}\dfrac1{|x|^{N+ps}} 
			\int_{B_{\frac{|x|}{2}}(0)}
			(\Upsilon(y)-\Upsilon(x))^{p-1}dy\\    
			&\le -\left(\dfrac{2}3\right)^{N+ps}\dfrac1{|x|^{N+ps}} 
			\int_{B_{1}(0)}
			(\Upsilon(y)-\Upsilon(x))^{p-1}dy\\   
			&\le -\left(\dfrac{2}3\right)^{N+ps}\dfrac1{|x|^{N+ps}} 
			\int_{B_{1}(0)}
			(\Upsilon(y)-\Upsilon(z))^{p-1}dy\\   
      \end{align*}
      for any $z\in\partial B_1(0).$  That is, there is a positive constant
      $C$ such that
      \begin{equation}\label{eq:Ip2}
      		I_4\le -\dfrac{C}{|x|^{N+ps}}\quad\forall |x|>2.
      \end{equation}
      Then by \eqref{eq:Ip1} and \eqref{eq:Ip2}, there is a positive constant
      $C$ such that 
      \[
		(-\Delta_p)^s\Upsilon(x)\le - \dfrac{C}{|x|^{N+sp}}      
      \]
      for all $x$ large enough and $\alpha$ such that $\alpha(p-1)>N.$
      
    \end{proof}

    The other result that play an important role in the proof of 
    decay estimates is the next comparison principle.

    \begin{theorem}[Comparison principle]\label{theorem:comparison}
        Let $\Omega\subset\mathbb{R}^N$ be an open set, 
        $V\in\lir$ $V\ge 0$ in $\Omega$ and  $u,v\in\wspr$ satisfy 
        $u\le v$ in $\mathbb{R}^N\setminus \Omega$ and
        \[
            (-\Delta_p)^s u+ V(x)|u(x)|^{p-2}u(x)\le
            (-\Delta_p)^s v+ V(x)|v(x)|^{p-2}v(x)
	                    \quad \text{ in }\Omega,
        \]
        that is
        \begin{align*}
                   \int_{\mathbb{R}^N}\int_{\mathbb{R}^N}&
                     \dfrac{|u(x)-u(y)|^{p-2}(u(x)-u(y))(\varphi(x)-\varphi(y))}
                     {|x-y|^{N+ps}}\,dxdy\\
                   &\hspace{4cm}+\int_{\mathbb{R}^N}
                   V(x)|u(x)|^{p-2}u(x)\varphi(x) dx\\
                   \le&\int_{\mathbb{R}^N}\int_{\mathbb{R}^N}
                     \dfrac{|v(x)-v(y)|^{p-2}(v(x)-v(y))(\varphi(x)-\varphi(y))}
                     {|x-y|^{N+ps}}\,dxdy\\
                   &\hspace{4cm}+\int_{\mathbb{R}^N}
                   V(x)|v(x)|^{p-2}v(x)\varphi(x) dx
        \end{align*}
        whenever $\varphi\in\twspo,$ $\varphi\ge0.$
        Then $u\le v$ in $\Omega.$
    \end{theorem}
    
    \begin{proof}
	    Let's start by observing that 
	    $\varphi=(u-v)_+\in\twspo,$ since
	     $u,v\in\wspr,$ and for any $x,y\in\mathbb{R}^N$ we have 
	    \[
	        |\varphi(x)|\le |u(x)-v(x)|,
	    \]    
	    \[
	        |\varphi (x)-\varphi (y)|\le |u(x)-v(x)-(u(y)-v(y))|\le 
	        |u(x)-u(y)|+ |v(x)-v(y)|.
	    \]
	    The proof follows by the argument of \cite[Proposition 2.5]{MR3631323}.
	    See also \cite[Lemma 9]{LL} and \cite[Theorem 2.6]{MR3461371}.
    \end{proof}
    
    First we show the  decay rate at infinity of all positive eigenfunctions 
    associated to $\lambda_1(g)$ in the case $sp\ge N.$

%
    \begin{proof}[Proof of Theorem \ref{theorem:decay1}]
	    We give only the proof of the left hand side of the inequality, 
	    the proof of the right hand side is similar.
	    
	    Let us observe that, by assumptions, we have
	    \begin{equation}\label{eq:decay1}
            0\le
            (-\Delta_p)^s u(x)+ \lambda_1(g)\|g_2\|_\infty|u(x)|^{p-1}
        \end{equation}
        for $|x|$ large enough. 
        
        On the other hand, taking $\alpha=\dfrac{N+ps}{p-1}$ and  
        $\Upsilon\in C^2(\mathbb{R}^N)$ a positive function such that 
        $\Upsilon$ is radially symmetric, decreasing and 
        \[
            \Upsilon(x)=\dfrac{1}{|x|^\alpha}\quad\forall |x|>1,
        \]
        by Lemma \ref{lemma:decay}, we have that there exists $k>>1$ such that 
        for any $|x|>k$
         \begin{equation}\label{eq:decay2}
              (-\Delta_p)^s \Upsilon(x)+ c_1|\Upsilon(x)|^{p-2}
              \Upsilon(x)\le0
        \end{equation}
        for some positive constants $c_1$ and $c_2.$
        
        We next set
        \[
            \phi(x)=K\Upsilon(Rx)
        \]
        where $K$ and $R$ are positive constant that will be selected bellow. 
        Then, by \eqref{eq:decay2},
        \[
            (-\Delta_p)^s \phi(x)=\dfrac{K}{R^{sp}}(-\Delta_p)^s \Upsilon(Rx)
            \le -\dfrac{c_1}{R^{sp}}|\phi(x)|^{p-2}\phi(x) \quad \forall |x|>\dfrac{k}{R}.
        \]
        Taking 
        \[
            R= \left(\dfrac{c_1}{\lambda_1(g)\|g_2\|_{\infty}}\right)^{\frac1{sp}}
             \quad\text{ and }\quad k_1=\dfrac{k}{R},
       \]
       we have
       \begin{equation}\label{eq:decay3}
            (-\Delta_p)^s \phi(x)+ \lambda_1(g)\|g_2\|_\infty|\phi(x)|^{p-2}\phi(x)\le 0
            \quad \forall |x|>k_1.
       \end{equation}
       
       Notice that $\phi(x)$ is classical solution and then a strong solution and weak solution, for details see \cite{MR3593528}.

       On the other hand,  
       by Lemmas \ref{lemma:PositiveFE3}, \ref{lemma:regularidad1} and
       \ref{lemma:regularidad2}, we can choose $K>0$ so that 
       $u(x)\ge\phi(x)$ in $|x|<k_1.$ 
        
       Finally, by Theorem \ref{theorem:comparison}, we get 
       $u(x)\ge\phi(x)$ in $\mathbb{R}^N.$ Therefore
       \[
            \dfrac{K}{R^{\frac{N+sp}{p-1}}}|x|^{-\frac{N+sp}{p-1}}\le u(x)
       \]
       for all $|x|>k_1.$
    \end{proof}
    
    Notice that in the case $sp<N$ one side bound can be obtained but the other is not posible 
    since the assumption $g(x)<-\delta<0$ for $|x|$ large enough is not compatible with the assumptions of the existence results.

    Lastly, we study the decay rate at infinity of all positive ground state 
    solutions of the next autonomous Schr\"odinger equations
    \begin{equation}\label{eq:asc1}
	    \begin{cases}
            (-\Delta_p)^su(x)+\mu|u|^{p-2}u=f(u) &\text{in }\mathbb{R}^N,\\
            u\in\wspr\\
            u(x)>0 &\text{for all }x\in\mathbb{R}^N, \mu>0.	    
	    \end{cases}
    \end{equation} 
    
    The existence of at least one positive ground state solution of \eqref{eq:asc1} 
    was recently proved in \cite{Ambrosio2018} under the following assumptions:
    $sp<N$ and the nonlinearity $f\colon\mathbb{R}\to\mathbb{R}$ satisfies the next 
    conditions
    \begin{enumerate}
	    \item[($f_1$)] $f\in C(\mathbb{R})$ and $f(t)=0$ for all $t<0;$ 
	    \item[($f_2$)] $\lim\limits_{|t|\to 0}\dfrac{|f(t)|}{|t|^{p-1}}=0;$
	    \item[($f_3$)] there is $q\in(p,p_s^*)$ such that
	    \[ 
	        \lim\limits_{|t|\to \infty}\dfrac{|f(t)|}{|t|^{q-1}}=0;
	    \]
	    \item[($f_4$)] there is $\vartheta>p$ such that
	    \[
	        0<\vartheta \int_0^t f(\tau) d\tau\le t f(t) \quad\text{ for all } t>0;
	    \]
	    \item[($f_5$)] the map $t\to\dfrac{f(t)}{t^{p-1}}$ is increasing in 
	    $(0,+\infty).$
    \end{enumerate}
    
    In addition, in \cite[Remark 3.2]{Ambrosio2018}, the authors observe that
    if $v$ is a positive ground state solutions of \eqref{eq:asc1} then 
    $v\in \lir\cap C(\mathbb{R}^N).$ 
    In fact, by \cite[Corollary 5.5]{MR3593528},
    we can conclude that $v\in C^{\gamma}(\mathbb{R}^N)$ for some 
    $\gamma\in(0,1).$ Therefore
    \begin{equation}\label{eq:aes}
	    v(x)\to 0\quad\text{ as } |x|\to0
    \end{equation}
    since $v\in \lpr\cap\lir\cap C^\gamma(\mathbb{R}^N).$ 
    
    Our last result shows the rate decay of $v$ at infinity.
    
    
    \begin{proof}[Proof of Theorem \ref{theorem:decay2}]
	    Observe that, by \eqref{eq:aes} and $(f_2)$ there is a $k>1$ such that
	    \[
	        f(v(x))<\dfrac{\mu}2|v(x)|^{p-2}v(x) \quad \forall |x|>k.
	    \]
	    Then
	    \[
            (-\Delta_p)^s v(x)+ \dfrac\mu2|v(x)|^{p-2}v(x) \le0\le
            (-\Delta_p)^s v(x)+ \mu|v(x)|^{p-2}v(x)
        \]
        for $|x|$ large enough.
        
        The remain of the proof is entirely analogous to the proof of Theorem
        \ref{theorem:decay1}.
    \end{proof}

{\bf Acknowledgements.} L. D. P. was partially supported by PICT2012 0153 from ANPCyT 
	(Argentina). A. Q. was partially supported by Fondecyt Grant No. 1151180 and Programa Basal, CMM. U. de Chile

\bibliographystyle{abbrv}
\bibliography{Biblio}

\def\cprime{$'$}
\begin{thebibliography}{10}

\bibitem{Adams}
R.~A. Adams.
\newblock {\em Sobolev spaces}.
\newblock Academic Press [A subsidiary of Harcourt Brace Jovanovich,
  Publishers], New York-London, 1975.
\newblock Pure and Applied Mathematics, Vol. 65.

\bibitem{MR1098396}
W.~Allegretto.
\newblock Principal eigenvalues for indefinite-weight elliptic problems in
  {${\bf R}^n$}.
\newblock {\em Proc. Amer. Math. Soc.}, 116(3):701--706, 1992.

\bibitem{MR1356326}
W.~Allegretto and Y.~X. Huang.
\newblock Eigenvalues of the indefinite-weight {$p$}-{L}aplacian in weighted
  spaces.
\newblock {\em Funkcial. Ekvac.}, 38(2):233--242, 1995.

\bibitem{Ambrosio2018}
V.~Ambrosio and T.~Isernia.
\newblock Multiplicity and concentration results for some nonlinear
  {S}chr\"{o}dinger equations with the fractional {$p$}-{L}aplacian.
\newblock {\em Discrete Contin. Dyn. Syst.}, 38(11):5835--5881, 2018.

\bibitem{Amghibech}
S.~Amghibech.
\newblock On the discrete version of {P}icone's identity.
\newblock {\em Discrete Appl. Math.}, 156(1):1--10, 2008.

\bibitem{MR2471902}
F.~Andreu, J.~M. Maz\'{o}n, J.~D. Rossi, and J.~Toledo.
\newblock A nonlocal {$p$}-{L}aplacian evolution equation with nonhomogeneous
  {D}irichlet boundary conditions.
\newblock {\em SIAM J. Math. Anal.}, 40(5):1815--1851, 2008/09.

\bibitem{BCF}
C.~Bjorland, L.~Caffarelli, and A.~Figalli.
\newblock Non-local gradient dependent operators.
\newblock {\em Adv. Math.}, 230(4-6):1859--1894, 2012.

\bibitem{MR3122168}
M.~Bonforte and J.~L. V\'{a}zquez.
\newblock Quantitative local and global a priori estimates for fractional
  nonlinear diffusion equations.
\newblock {\em Adv. Math.}, 250:242--284, 2014.

\bibitem{MR3461371}
L.~Brasco, S.~Mosconi, and M.~Squassina.
\newblock Optimal decay of extremals for the fractional {S}obolev inequality.
\newblock {\em Calc. Var. Partial Differential Equations}, 55(2):Art. 23, 32,
  2016.

\bibitem{MR3552458}
L.~Brasco and E.~Parini.
\newblock The second eigenvalue of the fractional {$p$}-{L}aplacian.
\newblock {\em Adv. Calc. Var.}, 9(4):323--355, 2016.

\bibitem{MR3411543}
L.~Brasco, E.~Parini, and M.~Squassina.
\newblock Stability of variational eigenvalues for the fractional
  {$p$}-{L}aplacian.
\newblock {\em Discrete Contin. Dyn. Syst.}, 36(4):1813--1845, 2016.

\bibitem{MR0269962}
F.~E. Browder.
\newblock Existence theorems for nonlinear partial differential equations.
\newblock In {\em Global {A}nalysis ({P}roc. {S}ympos. {P}ure {M}ath., {V}ol.
  {XVI}, {B}erkeley, {C}alif., 1968)}, pages 1--60. Amer. Math. Soc.,
  Providence, R.I., 1970.

\bibitem{MR1007489}
K.~J. Brown, C.~Cosner, and J.~Fleckinger.
\newblock Principal eigenvalues for problems with indefinite weight function on
  {${\bf R}^n$}.
\newblock {\em Proc. Amer. Math. Soc.}, 109(1):147--155, 1990.

\bibitem{MR1412438}
K.~J. Brown and N.~Stavrakakis.
\newblock Global bifurcation results for a semilinear elliptic equation on all
  of {$\bold R^N$}.
\newblock {\em Duke Math. J.}, 85(1):77--94, 1996.

\bibitem{Caffarelli2012}
L.~Caffarelli.
\newblock {\em Non-local Diffusions, Drifts and Games}, pages 37--52.
\newblock Springer Berlin Heidelberg, Berlin, Heidelberg, 2012.

\bibitem{MR3556755}
L.~M. Del~Pezzo and A.~Quaas.
\newblock Global bifurcation for fractional {$p$}-{L}aplacian and an
  application.
\newblock {\em Z. Anal. Anwend.}, 35(4):411--447, 2016.

\bibitem{MR3631323}
L.~M. Del~Pezzo and A.~Quaas.
\newblock A {H}opf's lemma and a strong minimum principle for the fractional
  {$p$}-{L}aplacian.
\newblock {\em J. Differential Equations}, 263(1):765--778, 2017.

\bibitem{DD}
F.~Demengel and G.~Demengel.
\newblock {\em Functional spaces for the theory of elliptic partial
  differential equations}.
\newblock Universitext. Springer, London, 2012.
\newblock Translated from the 2007 French original by Reinie Ern{\'e}.

\bibitem{DNPV}
E.~Di~Nezza, G.~Palatucci, and E.~Valdinoci.
\newblock Hitchhiker's guide to the fractional {S}obolev spaces.
\newblock {\em Bull. Sci. Math.}, 136(5):521--573, 2012.

\bibitem{MR3635980}
S.~Dipierro, M.~Medina, I.~Peral, and E.~Valdinoci.
\newblock Bifurcation results for a fractional elliptic equation with critical
  exponent in {$\Bbb{R}^n$}.
\newblock {\em Manuscripta Math.}, 153(1-2):183--230, 2017.

\bibitem{MR1336957}
P.~Dr\'{a}bek.
\newblock Nonlinear eigenvalue problem for {$p$}-{L}aplacian in {$\bold R^N$}.
\newblock {\em Math. Nachr.}, 173:131--139, 1995.

\bibitem{MR1390979}
P.~Dr\'{a}bek and Y.~X. Huang.
\newblock Bifurcation problems for the {$p$}-{L}aplacian in {${\bf R}^N$}.
\newblock {\em Trans. Amer. Math. Soc.}, 349(1):171--188, 1997.

\bibitem{MR3002595}
P.~Felmer, A.~Quaas, and J.~Tan.
\newblock Positive solutions of the nonlinear {S}chr\"{o}dinger equation with
  the fractional {L}aplacian.
\newblock {\em Proc. Roy. Soc. Edinburgh Sect. A}, 142(6):1237--1262, 2012.

\bibitem{FP}
G.~Franzina and G.~Palatucci.
\newblock Fractional {$p$}-eigenvalues.
\newblock {\em Riv. Math. Univ. Parma (N.S.)}, 5(2):373--386, 2014.

\bibitem{Grisvard}
P.~Grisvard.
\newblock {\em Elliptic problems in nonsmooth domains}, volume~24 of {\em
  Monographs and Studies in Mathematics}.
\newblock Pitman (Advanced Publishing Program), Boston, MA, 1985.

\bibitem{MR1364493}
Y.~X. Huang.
\newblock Eigenvalues of the {$p$}-{L}aplacian in {$\bold R^N$} with indefinite
  weight.
\newblock {\em Comment. Math. Univ. Carolin.}, 36(3):519--527, 1995.

\bibitem{MR3593528}
A.~Iannizzotto, S.~Mosconi, and M.~Squassina.
\newblock Global {H}\"older regularity for the fractional {$p$}-{L}aplacian.
\newblock {\em Rev. Mat. Iberoam.}, 32(4):1353--1392, 2016.

\bibitem{IN}
H.~Ishii and G.~Nakamura.
\newblock A class of integral equations and approximation of {$p$}-{L}aplace
  equations.
\newblock {\em Calc. Var. Partial Differential Equations}, 37(3-4):485--522,
  2010.

\bibitem{KORVENPAA2017}
J.~Korvenp{\"a}{\"a}, T.~Kuusi, and E.~Lindgren.
\newblock Equivalence of solutions to fractional p-laplace type equations.
\newblock {\em Journal de Math\'ematiques Pures et Appliqu\'ees}, 2017.

\bibitem{MR1101219}
H.~Kozono and H.~Sohr.
\newblock New a priori estimates for the {S}tokes equations in exterior
  domains.
\newblock {\em Indiana Univ. Math. J.}, 40(1):1--27, 1991.

\bibitem{LL}
E.~Lindgren and P.~Lindqvist.
\newblock Fractional eigenvalues.
\newblock {\em Calc. Var. Partial Differential Equations}, 49(1-2):795--826,
  2014.

\bibitem{Mazya}
V.~Maz'ya.
\newblock {\em Sobolev spaces with applications to elliptic partial
  differential equations}, volume 342 of {\em Grundlehren der Mathematischen
  Wissenschaften [Fundamental Principles of Mathematical Sciences]}.
\newblock Springer, Heidelberg, augmented edition, 2011.

\bibitem{MR3089742}
R.~Servadei.
\newblock The {Y}amabe equation in a non-local setting.
\newblock {\em Adv. Nonlinear Anal.}, 2(3):235--270, 2013.

\bibitem{MR3002745}
R.~Servadei and E.~Valdinoci.
\newblock Variational methods for non-local operators of elliptic type.
\newblock {\em Discrete Contin. Dyn. Syst.}, 33(5):2105--2137, 2013.

\bibitem{MR578914}
E.~Zeidler.
\newblock Lectures on {L}yusternik-{S}chnirelman theory for indefinite
  nonlinear eigenvalue problems and its applications.
\newblock In {\em Nonlinear analysis, function spaces and applications ({P}roc.
  {S}pring {S}chool, {H}orni {B}radlo, 1978)}, pages 176--219. Teubner,
  Leipzig, 1979.

\bibitem{MR768749}
E.~Zeidler.
\newblock {\em Nonlinear functional analysis and its applications. {III}}.
\newblock Springer-Verlag, New York, 1985.
\newblock Variational methods and optimization, Translated from the German by
  Leo F. Boron.

\end{thebibliography}

\end{document}